\documentclass[12pt]{article} 

\usepackage{amsmath, amssymb}
\usepackage{amsthm} 
\usepackage{url}
\usepackage{graphicx}
\usepackage{here}
\usepackage{mathrsfs}
\usepackage{comment}
\usepackage{fancyhdr}
\usepackage{enumitem}
\usepackage{tikz} 
\usetikzlibrary{intersections, calc, through}

\usepackage{sectsty}
\allsectionsfont{\normalfont}

\usepackage[
  bookmarks=true,
  bookmarksdepth=3,
  hyperfootnotes=false,
  bookmarksnumbered=true,
  bookmarksopen=true,
  colorlinks=true,
  linkcolor=blue,
  citecolor=red,
  pdftitle={Noncommutative Floquet--Bloch Theory for Nilpotent Groups: Representation-Theoretic Foundations},
  pdfsubject={Representation-theoretic foundations of noncommutative Floquet--Bloch theory for nilpotent lattices},
  pdfauthor={Atsushi Katsuda},
  pdfkeywords={Floquet--Bloch theory, nilpotent lattices, Kirillov orbit method, finite-dimensional representations, Pytlik formula, Harper operator}
]{hyperref}

\theoremstyle{plain} 
\newtheorem{theorem}{\indent\sc Theorem}[section]
\newtheorem{lemma}[theorem]{\indent\sc Lemma}

\newtheorem{proposition}[theorem]{\indent\sc Proposition}

\theoremstyle{definition} 
\newtheorem{definition}[theorem]{\indent\sc Definition}
\newtheorem{remark}[theorem]{\indent\sc Remark}
\newtheorem{example}[theorem]{\indent\sc Example}

\newtheorem{definitiontheorem}[theorem]{\indent\sc Definition-Theorem}

\newcommand{\Qodd}{\mathbb Q^{\mathrm{odd}}}
\newcommand{\widehatGamma}{\widehat{\Gamma}}

\newcommand{\Heis}{\mathrm{Heis}}

\makeatletter
    
    \@addtoreset{equation}{section}

\def\address#1#2{\begingroup
\noindent\parbox[t]{7.8cm}{%
\small{\scshape\ignorespaces#1}\par\vskip1ex
\noindent\small{\itshape E-mail address}%
\/: #2\par\vskip4ex}\hfill%
\endgroup}%

\makeatother
\title{Noncommutative Floquet--Bloch Theory for Nilpotent Groups: Representation-Theoretic Foundations} 
\author{%
Atsushi Katsuda$^{*}$\\
\emph{Dedicated to Professor Toshikazu Sunada}
}
\date{July 17, 2026}

\providecommand{\R}{\mathbb R}
\providecommand{\Z}{\mathbb Z}
\providecommand{\Q}{\mathbb Q}

\providecommand{\g}{\mathfrak g}
\providecommand{\m}{\mathfrak m}

\providecommand{\rL}{\mathfrak r}

\providecommand{\ee}{\mathrm e}
\providecommand{\ii}{\mathrm i}
\providecommand{\restr}{\big|}
\providecommand{\Ad}{\operatorname{Ad}}
\providecommand{\Ind}{\operatorname{Ind}}
\providecommand{\Tr}{\operatorname{Tr}}

\newtheorem*{theoremA}{\indent\sc Theorem A}
\newtheorem*{theoremB}{\indent\sc Theorem B}
\newtheorem*{theoremC}{\indent\sc Theorem C}

\begin{document}
\markright{Floquet--Bloch Decomposition for Nilpotent Groups}
\pagestyle{myheadings}
\maketitle

\begin{abstract}
Classical Floquet--Bloch theory decomposes abelian periodic problems
over the character torus of the lattice.  For nonabelian nilpotent
lattices, the non--type~I obstruction rules out a comparable
parametrization of the full unitary dual.  We do not attempt to remove
this obstruction.  Instead, we construct an exact Bloch-type replacement
on the representation-theoretic part of the theory which is visible from
rational Kirillov data and from finite-dimensional rational fibers.

Let $\Gamma$ be a torsion-free finitely generated nilpotent group and
let $G$ be its Malcev completion.  For an irreducible unitary
representation $\pi_l$ of $G$ attached to a rational Kirillov parameter
$l\in\mathfrak{g}_{\mathbb Q}^{*}$, we prove an exact restriction theorem
for $\pi_l|_\Gamma$.  The branching is first described by induced
representations attached to rational polarizations.  On the rational odd locus relevant to finite-dimensional
representations, it further decomposes into finite-dimensional irreducible
representations of $\Gamma$.

On these finite-dimensional rational fibers we construct a positive
finitely additive Plancherel measure.  It gives Fourier inversion and
normalized trace identities for nilpotent lattices, recovering Pytlik's
formula in the discrete Heisenberg case.
\end{abstract}

\footnotetext{$^{*}$This work was supported by JSPS KAKENHI Grant Number JP24K06715 and by the Research Institute for Mathematical Sciences, Kyoto University, an International Joint Usage/Research Center.}

\tableofcontents

\section{Introduction}

Classical Floquet--Bloch theory is a fundamental tool in the spectral
analysis of periodic operators.  For a \(\mathbb Z^d\)-periodic problem,
the regular representation of the lattice decomposes into one-dimensional
characters parametrized by the dual torus.  This converts a global problem
into a family of fiberwise problems, and the spectral information is encoded
in band functions.

For nonabelian covering groups this mechanism has no literal analogue.
Even for finitely generated nilpotent groups, the full unitary dual is
generally too singular to be used as a manageable Bloch parameter space.
Glimm's theorem relates standard-Borel classification of irreducible
representations to the type~I property \cite{Glimm}.  Thoma's theorem gives
the corresponding group-theoretic criterion for countable discrete groups:
such a group is of type~I exactly when it contains an abelian subgroup of
finite index \cite{ThomaUnitary}.  Thus the nonabelian torsion--free finitely generated nilpotent groups
considered here are non--type~I, and their full irreducible duals cannot
serve as Bloch parameter spaces comparable with the character torus.  This
obstruction remains fully in force throughout the paper.

The point of the present work is different.  We do not try to diagonalize
the regular representation by resolving the full unitary dual.  Instead, we
use the Malcev completion of the lattice and the Kirillov orbit method on
the Lie group side \cite{CorwinGreenleaf,Kirillov1962,Kirillov2004},
and then restrict the resulting representations back to the lattice.  For
rational Kirillov parameters, the restriction admits an exact branching
formula over a compact torus.  On a rational period subfamily, described using
Howe's classification of finite-dimensional representations of nilpotent
lattices \cite{Howe}, the lattice fibers admit a further finite-dimensional
rational refinement.  Together with a
Pytlik-type finitely additive inversion functional \cite{Pytlik}, these
fibers give an exact noncommutative Floquet--Bloch skeleton for nilpotent
lattices.

It is important to keep three objects separate.  The first restriction
theorem is an ordinary Hilbert direct integral with respect to Lebesgue
measure on a torus.  Its fibers are induced lattice representations and need
not all be finite-dimensional.  The finite-dimensional rational refinement is a further arithmetic
refinement on a rational period locus.  The Pytlik-type
inversion formula is a separate finitely additive trace identity supported
on finite-dimensional representations.  Confusing these three levels would
obscure the role of the non--type~I obstruction and lead to an incorrect
interpretation of the main theorem.

\subsection*{Rational structure and Malcev completion}

Let \(\Gamma\) be a torsion-free finitely generated nilpotent group.  Its
Malcev completion is a connected, simply connected nilpotent Lie group \(G\)
containing \(\Gamma\) as a lattice.  Conversely, by Malcev's theorem, such a
Lie group admits a lattice precisely when its Lie algebra \(\g\) is defined
over \(\Q\); equivalently, \(\g\) has a basis \(X_1,\ldots,X_n\) in which
\[
   [X_i,X_j]=\sum_{k=1}^n c_{ij}^k X_k,\qquad c_{ij}^k\in\Q .
\]
We use this standard form of Malcev theory as recalled, for example, in
\cite[Theorem~5.1.8]{CorwinGreenleaf}.  The lattice determines a rational structure
\(\g_\Q\).  In general, however, \(\log \Gamma\) itself should not be
treated as a Lie algebra or as an additive lattice.  One works instead with
Malcev bases and with the rational span determined by the lattice; this
point is important in the finite-dimensional part of the argument.

For \(l\in\g^*\), let
\[
   \rL_l=\{X\in\g\mid l([X,Y])=0\text{ for all }Y\in\g\}
\]
be the radical of the skew form \(B_l(X,Y)=l([X,Y])\).  A polarization for
\(l\) is a maximal subordinate subalgebra \(\m\subset\g\), i.e. a maximal
Lie subalgebra satisfying \(l([\m,\m])=0\).  If \(l\in\g^*_\Q\), then
\(\rL_l\) is rational and a rational polarization may be chosen
\cite{CorwinGreenleaf,Vergne}.

\subsection*{Main results}

The structural result of the paper is the following restriction theorem.  It
is the result that plays the role of the Bloch decomposition.

\begin{theoremA}[Restriction theorem]
Let \(\Gamma\) be a torsion-free finitely generated nilpotent group, let
\(G\) be its Malcev completion, and let \(l\in\g^*_\Q\).  Let \(\pi_l\) be
the irreducible unitary representation of \(G\) attached to the coadjoint
orbit of \(l\).  Then one can choose a rational polarization and a
lattice-adapted weak Malcev system such that
\[
   \pi_l\restr_\Gamma
   \simeq
   \int_{[0,1)^{n-m}}^\oplus
      \Ind_{M_t\cap\Gamma}^{\Gamma}(\chi_t)\,
      dt_{m+1}\cdots dt_n .
\]
Here \(m=\dim M_t\), the parameter \(t=(t_{m+1},\ldots,t_n)\) belongs to a
compact torus, and the inducing data \((M_t,\chi_t)\) are obtained
explicitly from the Kirillov orbit method.  In the fixed coordinates used in
the proof this is written as
\[
   \pi_l\restr_\Gamma
   \simeq
   \int_{[0,1)^{n-m}}^\oplus
      \Ind_{M\cap\Gamma}^{\Gamma}
      \bigl(\chi_{l_t}\restr_{M\cap\Gamma}\bigr)\,dt,
\]
where
\[
   l_t=\Ad^*(\exp(t_{m+1}X_{m+1}))\cdots
       \Ad^*(\exp(t_nX_n))l,
\]
and the notation includes the harmless polarization transport used in the
non-ideal induction step.
\end{theoremA}

The first theorem is an exact restriction statement.  The next theorem is
the finite-dimensional refinement.  It is deliberately stated with its
rationality hypotheses visible.

\begin{theoremB}[Finite-dimensional rational refinement]
In the setting of Theorem~A, suppose that the parameter \(t\) belongs to the
rational odd locus \(([0,1)\cap\Qodd)^{n-m}\) and satisfies the corresponding
odd-period condition used in Howe's classification.  Let \(P_t\) be a discrete polarization for the
transported functional on the lattice and let \(\chi_t^P\) be an
extension of the character \(\chi_t = \chi_t^\Gamma\) from \(M_t\cap\Gamma\) to \(P_t\).  Then
\[
   \Ind_{M_t\cap\Gamma}^{\Gamma}(\chi_t)
   \simeq
   \int_{\widehat{P_t/(M_t\cap\Gamma)}}^\oplus
      \Ind_{P_t}^{\Gamma}(\chi_t^P\otimes \eta)\,d\eta .
\]
Each representation in the integrand is finite-dimensional and irreducible
of the type classified by Howe.
\end{theoremB}

Theorem~B is the precise form of ``finite-dimensionalization''.  It should
not be read as saying that every Lebesgue-generic fiber in Theorem~A is
finite-dimensional.  Rather, it identifies the finite-dimensional rational
fibers on the rational period subfamily, using the classification theorem of Howe;
this is the locus used by the Pytlik-type inversion functional.  This theorem is the finite-dimensional
skeleton used in the companion analytic paper: all lattice-side traces there
are normalized finite-dimensional traces, while Kirillov or Schr\"odinger
models are used only as smooth normal forms for computing large-height
coefficients.

The finite-dimensional approximation should be understood in the following
limited large-denominator sense.  At rational parameters the
finite-dimensional rational fibers are exact Bloch fibers.  At irrational
parameters, rational characters with large denominators form a dense
arithmetic skeleton in the relevant compact abelian parameter spaces and
therefore provide nearby exact rational models.  The present paper does not
assert a uniform operator-norm or smooth graph-norm convergence theorem
comparing those fibers with an irrational Kirillov or Schr\"odinger model.
Such quantitative estimates depend on the chosen identifications and on the
observables being compared, and are left to the companion analytic work.
For comparison, finite-dimensional approximation in the Fell topology and
related results in residually finite and amenable settings are discussed by
Lubotzky--Shalom \cite{Lubotzky}.  The rigorous statements used below are the
exact rational decompositions and the finite-dimensional Pytlik-type trace
identities.

Combining the factor-Plancherel decomposition of the regular representation
of \(\Gamma\), the extension of the corresponding factor representations to
monomial representations of \(G\), Fujiwara's Plancherel formula
\cite{FujiwaraPlancherel,FujiwaraNote,Corwin2}, and Theorems~A and~B, we obtain the
following generalization of Pytlik's formula for the discrete Heisenberg
group.

\begin{theoremC}[Generalized Pytlik formula]
Let \(\Gamma\) be a torsion-free finitely generated nilpotent group.  There
exists a positive finitely additive probability measure \(\mu\) on
\(\widehatGamma\), supported on the finite-dimensional irreducible
representations \(\widehatGamma_{\mathrm{fin}}\), such that for every
\(f\in\ell^1(\Gamma)\) and every \(\sigma\in\Gamma\),
\[
   f(\sigma)=
   \int_{\widehatGamma_{\mathrm{fin}}}
      \frac{1}{\dim \pi_{\text{\tiny fin}}}
      \Tr\bigl(\pi_{\text{\tiny fin}}(\sigma^{-1})\pi_{\text{\tiny fin}}(f)\bigr)\,
      d\mu(\pi_{\text{\tiny fin}}).
\]
The associated Plancherel identity is also valid.  The measure is used only
for Fourier inversion and trace identities; it is not an ordinary
Hilbert-space direct-integral measure.
\end{theoremC}

In this form the role of the non--type~I obstruction is transparent.  We do
not diagonalize the regular representation by a standard Borel unitary dual.
Instead, the Malcev completion supplies explicit representation-theoretic
models, and finite-dimensional traces give the lattice-side inversion
formula.

\subsection*{The Heisenberg prototype}

Let
\[
   \Gamma=\Heis_3(\Z),\qquad G=\Heis_3(\R).
\]
For \(h=p/q\in\Q\) with \((p,q)=1\), the Schr\"odinger representation
\(\rho_h\) of \(G\) restricts to \(\Gamma\) as
\[
   \rho_h\restr_{\Heis_3(\Z)}
   \simeq
   \int_0^1{}^\oplus\int_0^1{}^\oplus
      \rho_{\text{\tiny fin},(h,\{qhx_2\},x_3)}\,dx_2\,dx_3 .
\]
Thus the finite-dimensional representations which appear in Pytlik's formula
also appear as exact fibers of the restriction of an infinite-dimensional
Lie-group representation.  This is the model for the general nilpotent
construction.

Applied to the group-algebra element \(u+u^{-1}+v+v^{-1}\), the same formula
gives the usual rational Harper matrices as exact finite-dimensional fibers.
The Wilkinson expansion is then an asymptotic statement inside these exact
\(q\)-dimensional fibers, not a justification for their existence.

\subsection*{Relation with earlier work}

The orbit-method background is due to Kirillov, Mackey, Vergne, and the
standard accounts of Corwin--Greenleaf
\cite{Kirillov1962,Kirillov2004,Mackey,Vergne,CorwinGreenleaf}.  The first
branching theorem should be read in direct relation with earlier restriction
theory.  In the ideal-polarization case it follows the argument of
Bekka--Driutti \cite{BekkaDriutti}; the author learned of this paper from
Ali Baklouti.  The additional polarization-replacement step is the standard
Kirillov--Corwin--Greenleaf mechanism which allows one to pass from a
non-ideal polarization to an equivalent one adapted to a codimension-one
ideal.  Thus the novelty is not a new proof of the classical
ideal-polarization restriction theorem, but the way this restriction theory
is organized with rational lattice data, weak Malcev coordinates, measurable
Bloch families, the classification theorem of Howe, and the final
Pytlik-type inversion formula.

The construction also uses earlier Plancherel theory in an essential way.
The factor-Plancherel decomposition for the lattice side is related to the
work of Mautner, Segal, Thoma, Johnston, and Bekka
\cite{Mautner,Segal,ThomaRegulareMass,ThomaUnitary,Johnston,BekkaPlancherel}.
Fujiwara's Plancherel formula for monomial representations, together with
the Corwin--Greenleaf--Gr'elaud decomposition theory and his note
communicated to the author, supplies the bridge from the factor
representations extended to the Malcev completion to ordinary Kirillov
representations of the Lie group \cite{FujiwaraPlancherel,Corwin2,FujiwaraNote}.
The elementary extension of lattice factor representations to the Malcev
completion by means of the Baker--Campbell--Hausdorff formula was pointed out
to the author by Takashi Iwamoto before the present project was developed.
Finite-dimensional representations of nilpotent lattices are then governed
by Howe's classification \cite{Howe}.

The discrete Heisenberg case is inspired by Pytlik's finitely additive
Plancherel formula \cite{Pytlik}.  The analytic model involving the Harper
operator is related to the work of Wilkinson, Rammal--Bellissard,
Helffer--Sj\"ostrand, Choi--Elliott--Yui, and Avila--Jitomirskaya
\cite{Wilkinson,RammalBellissard,HelfferSjostrandI,HelfferSjostrandII,ChoiElliottYui,AvilaJitomirskaya}.

Several works under the name noncommutative Bloch theory, including those of
Mathai--Marcolli and Gruber, approach related problems through
noncommutative geometry, twisted index theory, \(C^*\)-algebras, or
qualitative spectral invariants \cite{MathaiMarcolli,Gruber}.  The present
paper has a different purpose: it produces a concrete
representation-theoretic restriction and finite-dimensional trace mechanism
for nilpotent lattices.

\subsection*{Organization of the paper}

Section~\ref{dHeisenberg} develops the Heisenberg model, including the
finite-dimensional Pytlik fibers and the direct two-step proof of the
restriction formula.  Section~\ref{Representationsdiscretenilpotent} gives
the nilpotent representation-theoretic construction: factor representations,
monomial extension to the Malcev completion, Fujiwara's Plancherel theorem,
the first restriction theorem, finite-dimensional rational refinement via Howe's classification,
and the generalized Pytlik formula.  Section~\ref{constracthypoelliptic}
records the construction of the associated canonical hypo-elliptic operator,
and Section~\ref{Harperfoundation} explains the limited foundational role of
the Harper operator as an exact finite-dimensional fiber model.  The
appendices are mainly a reference section: they recall induced
representations and standard orbit-method facts for nilpotent Lie groups.
The finite-quotient material included there is only a record of the
trace-functional convention used in Section~\ref{Representationsdiscretenilpotent};
it is not a second proof of the finite-dimensional decomposition.

\subsection*{Use of AI tools}

AI and LLM tools were used as editorial and checking assistants: for
copy-editing, notation and cross-reference checks, and assistance with
\LaTeX{} preparation.  They were not used as independent sources of
mathematical results.  All mathematical definitions, statements, proofs,
corrections, and final judgments remain the responsibility of the author.

\subsection*{Note on earlier versions}

Preliminary versions of portions of this work appeared in the earlier
integrated preprint arXiv:2509.16848.  The statements and proofs of the
results treated in the present paper should be read in the form given here.

\subsection*{Acknowledgements}

First and foremost, the author is deeply grateful to Hidenori Fujiwara for
his expert guidance on the representation theory of nilpotent Lie groups,
which lies at the heart of this project.  In particular, his insightful
note~\cite{FujiwaraNote} and his comprehensive article~\cite{FujiwaraPlancherel}---as
discussed in Step N-3 (Section~\ref{Step0-3})---have been indispensable.

The author also wishes to thank Ali Baklouti, Fumio Hiroshima, Takashi Iwamoto, Tsuyoshi Kajiwara, Hisashi
Naito, Hiroyuki Ochiai, and Tatsuya Tate for many stimulating discussions and
valuable suggestions.

Last but not least, the author thanks Toshikazu Sunada for introducing him to
these themes and for his longstanding advice and encouragement.

\section{The Heisenberg Group as a Model Case}\label{dHeisenberg}\label{sec:heisenberg}
We begin with the three-dimensional discrete Heisenberg group
\(
\Gamma=\mathrm{Heis}_3(\mathbb Z)
\),
because in this case the finite-dimensional fibers and their relation to the
Schr\"odinger representation of \(\mathrm{Heis}_3(\mathbb R)\) can be written
explicitly.  The purpose of this section is to present the direct Heisenberg
construction without using the orbit method.  This direct construction is
computationally transparent and was the original route to the decomposition
formula used in this paper.

The orbit-method formulation for general nilpotent lattices is developed in
Section~\ref{Representationsdiscretenilpotent}.  There the Heisenberg example
is revisited as the first model for the general restriction theorem; comparing
the two presentations shows that the elementary Bloch decomposition here and
the later orbit-method formulation describe the same representation-theoretic
content.  Analytic applications of the same finite-dimensional fibers, such as
Harper-type operators, are only indicated in this foundational paper. Their details and, applications to heat-kernel and closed geodesic asymptotics, are treated in the companion analytic paper.

\subsection{Unitary Representations of the Discrete Heisenberg Group}
\label{discreteHeisenberg}

Let us begin with the three--dimensional discrete Heisenberg group
\(
\Gamma = \mathrm{Heis}_3(\mathbb{Z})
\).
We recall the geometric formulation of Floquet--Bloch theory for a general discrete
group $\Gamma$.

Let $\pi \colon X \to M$ be a normal covering of a compact Riemannian manifold $M$
with covering transformation group $\Gamma$.
The space $L^2(X)$ can be identified with the space $L^2(E_R)$ of square--integrable
sections of the flat vector bundle $E_R$ associated with the right regular
representation $R$ of $\Gamma$.
A central problem is to decompose $L^2(E_R)$ into a direct integral of spaces
associated with irreducible unitary representations of $\Gamma$.

Since $\Gamma$ is not of type~I, the natural Mackey Borel structure on its
irreducible dual $\widehat{\Gamma}$ is not standard.  Consequently,
$\widehat{\Gamma}$ cannot be used as a standard-Borel Bloch parameter space,
and there is no canonical irreducible Plancherel decomposition of the usual
type~I form.  To overcome this difficulty, we employ a version of the
Plancherel theorem due to
Pytlik~\cite{Pytlik}, which provides a Fourier inversion formula based on
finite--dimensional representations.

Let $f \in L^1(\Gamma)$ and let $\pi$ be a unitary representation of $\Gamma$.
The Fourier transform $\pi(f)$ is defined by
\[
\pi(f) = \int_\Gamma f(\sigma)\pi(\sigma)\, d\sigma,
\]
where $d\sigma$ denotes the counting measure on $\Gamma$.

\begin{theorem}[Pytlik~\cite{Pytlik}]\label{Pytlik}
There exists a positive, finitely additive measure $\mu$ on $\widehat{\Gamma}$,
supported on the set $\widehat{\Gamma}_{\mathrm{fin}}$ of finite--dimensional
irreducible unitary representations of $\Gamma$, such that for any
$f \in L^1(\Gamma)$,
\begin{align*}
\sum_{n \in \Gamma} |f(n)|^2
=
\int_{\widehat{\Gamma}_{\mathrm{fin}}}
\frac{1}{\dim \rho}\,
\mathrm{Tr}\!\left(\rho(f)^* \rho(f)\right)
\, d\mu(\rho).
\end{align*}
\end{theorem}

As a consequence, for $f \in \ell^1(\Gamma)$ and $\sigma \in \Gamma$,
we obtain the following Fourier inversion formula:
\begin{equation}
f(\sigma)
=
\int_{\widehat{\Gamma}_{\mathrm{fin}}}
\frac{1}{\dim \rho}\,
\mathrm{Tr}\!\left(\rho(\sigma^{-1})\rho(f)\right)
\, d\mu(\rho).
\label{Fourier1}
\end{equation}

In the case of $\Gamma = \mathrm{Heis}_3(\mathbb{Z})$,
the dual $\widehat{\Gamma}_{\mathrm{fin}}$ may be identified with
\[
\widehat{X} := (\mathbb{Q}\cap (0,1]) \times [0,1) \times [0,1),
\]
where the first coordinate is read modulo $1$ with $1$ representing the
trivial central character, and the last two coordinates are also read modulo
$1$.  The measure $\mu$ is given by
\begin{equation}
d\mu(x_1,x_2,x_3)
=
d\widetilde{m}(x_1)\, dm(x_2)\, dm(x_3),
\label{measure}
\end{equation}
where $m$ is Lebesgue measure on $[0,1)$,
and $\widetilde{m}$ is the finitely additive measure satisfying
\(
\widetilde{m}(\mathbb{Q}\cap(a,b]) = b-a
\)
for $0\leq a<b\leq1$.

The finite--dimensional irreducible unitary representations
$\rho=\rho_{\mathrm{fin},x}$ are described explicitly in
\cite{Pytlik}.
We briefly recall the construction
in a form convenient for later use.

Let $x=(x_1,x_2,x_3)\in\widehat{X}$,
and assume that $x_1=p/q$ is written in lowest terms.
Then $\rho_{\mathrm{fin},x}$ acts on
$\mathcal{H}_x \simeq \mathbb{C}^q$ by
\begin{equation}
(\rho_{\mathrm{fin},x}(n)\varphi)(k)
=
\exp\!\left(
\frac{2\pi i}{q}
(n_3 x_3 + n_2 x_2 + n_1 q x_1 + k n_2 q x_1)
\right)
\varphi(k+n_3),
\label{finite-rep}
\end{equation}
where
\begin{equation}
n=[n_1,n_2,n_3]
=
\begin{pmatrix}
1 & n_3 & n_1\\
0 & 1 & n_2\\
0 & 0 & 1
\end{pmatrix}
\in\Gamma.
\label{3heis}
\end{equation}

\medskip

There is an equivalent and more concrete matrix description
(cf.~\cite{Davidson}),
which makes the structure of the representation transparent.

Let
\[
w=[1,0,0],\qquad
v=[0,1,0],\qquad
u=[0,0,1]
\in \Gamma,
\]
and set
\begin{equation}
\gamma = \exp(2\pi i x_1) = \exp\Bigl(2\pi i \frac{p}{q}\Bigr),
\qquad
\alpha = \exp\!\left(\frac{2\pi i}{q}x_3\right),
\qquad
\beta  = \exp\!\left(\frac{2\pi i}{q}x_2\right)
\in U(1). \label{fluctuation}
\end{equation}

Then $\rho_{\mathrm{fin},x}$ is determined by
\[
\rho_{\mathrm{fin},x}(w) = \gamma I_q,
\]
and
\begin{equation}
\label{matrixrep}
\rho_{\mathrm{fin},x}(u)
=
\alpha
\begin{pmatrix}
0 & 1 & 0 & \cdots & 0\\
0 & 0 & 1 & \ddots & \vdots\\
\vdots & \ddots & \ddots & \ddots & 0\\
0 & \cdots & 0 & 0 & 1\\
1 & 0 & \cdots & 0 & 0
\end{pmatrix},
\qquad
\rho_{\mathrm{fin},x}(v)
=
\beta
\begin{pmatrix}
1 & 0 & \cdots & 0\\
0 & \gamma & \cdots & 0\\
\vdots & \ddots & \ddots & \vdots\\
0 & \cdots & 0 & \gamma^{q-1}
\end{pmatrix}.
\end{equation}

Here we identify $\mathcal{H}_x$ with $\mathbb{C}^q$ using column vectors
indexed by $k=0,\ldots,q-1$.  Thus the first matrix realizes
$\varphi(k)\mapsto\varphi(k+1)$, and the matrices satisfy
$\rho(u)\rho(v)=\gamma\rho(v)\rho(u)$, in agreement with $uv=wvu$.

\medskip

This description immediately makes visible
the dependence of the representation
on the parameters $\alpha$ and $\beta$,
which control the size of fluctuations
in later decomposition formulas.

The central scalar $\gamma$ is an invariant of the representation.  For a
fixed $\gamma$ (and hence a fixed order $q$), the pair
$(\alpha^q,\beta^q)\in U(1)\times U(1)$ is uniquely determined by
$\rho_{\mathrm{fin},x}$ and determines it up to unitary equivalence.
Equivalently,
\begin{equation}
\label{equivalence}
\text{the correspondence }
\rho_{\mathrm{fin},x}
\longleftrightarrow
(\gamma,\alpha^q,\beta^q)
\text{ is bijective up to unitary equivalence.}
\end{equation}

Independently of the finitely additive measure above, the countably additive
Zak--Floquet transforms used below give ordinary Hilbert direct-integral
decompositions at each fixed rational central character, and compatible
periodic operators decompose fiberwise.  However, since the representation
spaces depend on the parameters $x$,
a comparison with continuous models is required in order to carry out
perturbative analysis.

To this end, we now relate the finite--dimensional representations
$\rho_{\mathrm{fin},x}$ to the Schr\"odinger representations of the real
Heisenberg group.

\subsection{Relation with the Real Heisenberg Group}\label{relationdiscLie}

We now recall the irreducible unitary representations of the real Heisenberg group

\begin{equation}
G = {\rm Heis}_3(\mathbb{R}) = \left\{ (z, y, x) := \begin{pmatrix} 1 & x & z \\ 0 & 1 & y \\ 0 & 0 & 1 \end{pmatrix} \,\middle|\, x, y, z \in \mathbb{R} \right\}. \label{Matrixrepresentation}
\end{equation}

These representations fall into two classes (cf.~\cite{CorwinGreenleaf}):

\begin{description}
\item[(1)] One-dimensional characters that are trivial on the center of $G$.

\item[(2)] Infinite-dimensional irreducible representations $\rho_h$, known as the Schr\"odinger representations, parametrized by $h \in \mathbb{R} \setminus \{0\}$. For $f \in L^2(\mathbb{R})$ and $\gamma = (z, y, x) \in G$, the action is given by
\begin{equation}
(\rho_h(\gamma)f)(s) = e^{2\pi i h(z + s y)} f(s + x). \label{schrep}
\end{equation}
After a unitary dilation, this is equivalent to
\begin{equation}
(\rho_h(\gamma)f)(s)
=
e^{2\pi i\left(h z+\operatorname{sgn}(h)\sqrt{|h|}\,s y\right)}
e^{\sqrt{|h|}x\frac{d}{ds}}f(s),
\label{srep}
\end{equation}
where the exponential of the differential operator denotes the unitary
translation group.
\end{description}

\begin{remark}
The expression~\eqref{schrep} differs from that in Example 2.2.6 (4) of~\cite{CorwinGreenleaf}. This discrepancy arises from differing conventions for the exponential coordinates on the Lie algebra ${\rm Lie}({\rm Heis}_3(\mathbb{R}))$. The formula~\eqref{schrep} corresponds to the matrix realization~\eqref{Matrixrepresentation}, while the expression in~\cite{CorwinGreenleaf} uses canonical coordinates of the first kind. See Subsubsection~\ref{2n+1dimensionalHeisenberg} for a detailed comparison.
\end{remark}

The Schr\"odinger representations play a central role in the Plancherel decomposition of $L^2(G)$. The following Fourier inversion formula holds (cf.~\cite{CorwinGreenleaf}):

\begin{theorem}\label{Fourierinversionheisenberglie}
Let $\mathcal{S}(G)$ denote the space of rapidly decreasing functions on $G$. For any $f \in \mathcal{S}(G)$ and $\sigma \in G$, we have
\begin{equation}
f(\sigma) = \int_{\mathbb{R} \setminus \{0\}} {\rm Tr}\left( \rho_h(\sigma^{-1}) \rho_h(f) \right) |h|\, dh, \label{Liedecomp}
\end{equation}
where $\rho_h(f)$ is the Fourier transform of $f$:

\[
\rho_h(f) = \int_G f(\gamma)\, \rho_h(\gamma)\, d\nu(\gamma),
\]

and $d\nu$ is a left-invariant Haar measure on $G$.
\end{theorem}

\begin{remark}
If $f \in L^1(G)$, then $\rho_h(f)$ is a compact operator (cf. Corollary 6.2.16 in~\cite{Fujiwara}), though not necessarily of trace class. For further discussion, see Appendix~D.
\end{remark}

We now describe how the Schr\"odinger representations $\rho_h$ restrict to the discrete subgroup $\Gamma = {\rm Heis}_3(\mathbb{Z})$. When $h = p/q \in \mathbb{Q}$ (with $p$ and $q$ relatively prime), the restriction of $\rho_h$ to $\Gamma$ decomposes into a direct integral of finite-dimensional representations:

\begin{theorem}\label{discretetoLie}
Let $h = x_1 = p/q \in \mathbb{Q} \cap (0,1]$. Then the restriction of the Schr\"odinger representation $\rho_h$ to $\Gamma$ decomposes as
\begin{equation}
\rho_h\big|_{\Gamma} \simeq \int_0^1 \!\! \int_0^1 \rho_{{\rm fin}, (x_1, \{q x_1 x_2\}, x_3)}\, dm(x_2)\, dm(x_3), \label{discretetoHeisenbergLie}
\end{equation}
where $\{a\}$ denotes the fractional part of $a \in \mathbb{R}$, and $dm$ is the Lebesgue measure on $[0,1)$.
\end{theorem}

This result is the Heisenberg prototype of the general restriction theorem, Theorem~\ref{sec3decompose}.  Here, however, we keep the direct proof, avoiding the orbit method, because it was the original route to the decomposition and serves as a model for the later nilpotent argument.

We divide the proof into two steps, both of which are essentially abelian in nature.  We label them H-1 and H-2 to distinguish the direct Heisenberg construction from the nilpotent steps in Section~\ref{Representationsdiscretenilpotent}.  Step H-1 is the finite-dimensional rational Bloch-fiber decomposition, and Step H-2 is the direct real-line slicing proof of the restriction formula.  These steps are kept here since they were the author's original direct route to the Heisenberg theorem and are different from the proof of the general nilpotent theorem given in the later section. 

The basic strategy is simple, based on the following correspondences:

\begin{description}
\item[Structure of Step H-1]
\begin{align}
L^2(\mathbb{Z}) &\simeq L^2((\mathbb{Z}/q\mathbb{Z}) \times q\mathbb{Z}) \simeq L^2(\mathbb{Z}/q\mathbb{Z}) \otimes L^2(q\mathbb{Z}) \notag \\
&\simeq L^2(\mathbb{Z}/q\mathbb{Z}) \otimes L^2(S^1) \simeq L^2(\mathbb{Z}/q\mathbb{Z}) \otimes L^2([0,1)). \label{step1discretetolie}
\end{align}

\item[Structure of Step H-2]
\begin{equation}
L^2(\mathbb{R}) \simeq L^2((\mathbb{R}/\mathbb{Z}) \times \mathbb{Z}) \simeq L^2([0,1)) \otimes L^2(\mathbb{Z})\left(\simeq L^2(E_R(\mathbb{R}/\mathbb{Z}))\right), \label{step2discretetolie}
\end{equation}
where $R$ is the right regular representation of $\mathbb{Z}$, and $L^2(E_R(\mathbb{R}/\mathbb{Z}))$ denotes the space of $L^2$-sections of the flat vector bundle over the circle $S^1 = \mathbb{R}/\mathbb{Z}$ associated to $R$.

\end{description}

\medskip

\begin{proof}
\noindent
Step H-1. We begin by extending the action of $\rho_{{\rm fin},x}$ on $\mathcal{H}_x$ to an action on $L^2(\mathbb{Z})$ using the same formula as in~\eqref{finite-rep}. Denote this extended action by $\rho_{L^2(\mathbb{Z}),x}$. Then, corresponding to \ref{step1discretetolie}, we have the decomposition
\begin{equation}
L^2(\mathbb{Z}) \simeq \int_{\widehat{q\mathbb{Z}}}^\oplus \mathcal{H}_{\chi,q}^{\mathbb{Z}}\, d\chi, \label{Zdecomp}
\end{equation}
where $\widehat{q\mathbb{Z}}$ is the unitary dual of $q\mathbb{Z}$, and for
$\chi = \exp(2\pi i a) \in U(1)$ with $a \in [0,1)$, we define the
$q$-dimensional Hilbert space
\[
\mathcal{H}_{\chi,q}^{\mathbb{Z}}
:=
\left\{ \varphi: \mathbb{Z} \to \mathbb{C}
\,\middle|\, \varphi(k+q) = \chi \varphi(k) \right\},
\qquad
\|\varphi\|_{\chi,q}^2
:=
\sum_{k=0}^{q-1}|\varphi(k)|^2.
\]

Each $\mathcal{H}_{\chi,q}^{\mathbb{Z}}$ is invariant under $\rho_{L^2(\mathbb{Z}),x}$, and its restriction is unitarily equivalent to $\rho_{{\rm fin},(x_1, x_2, x_3 + a)}$ on $\mathcal{H}_{(x_1, x_2, x_3 + a)}$. This equivalence is realized via the identification

\[
\mathcal{H}_{\chi,q}^{\mathbb{Z}} \ni s \varphi \longleftrightarrow \varphi \in \mathcal{H}_{(x_1, x_2, x_3 + a)},
\]

where $s(k) := \chi^{k/q} = \exp(2\pi i a k / q)$ acts as a multiplication operator. This identification is compatible with the action of $\Gamma$ due to the relations

\[
\rho_{{\rm fin},x}(u) s = \exp(2\pi i a/q) s \rho_{{\rm fin},x}(u), \quad \rho_{{\rm fin},x}(v) s = s \rho_{{\rm fin},x}(v).
\]

Thus, we obtain

\[
\rho_{L^2(\mathbb{Z}),x} \simeq \int_0^1{}^\oplus \rho_{{\rm fin},(x_1, x_2, x_3 + a)}\, dm(a).
\]

\medskip
\noindent
Step H-2. Next, we extend the action of $\rho_{L^2(\mathbb{Z}),x}$ to $L^2(\mathbb{R})$ using the same formula~\eqref{finite-rep}, now interpreted for $s \in \mathbb{R}$. Denote this action by $\rho_{L^2(\mathbb{R}),x}$.

Corresponding to \ref{step2discretetolie}, we consider the decomposition

\begin{equation}
L^2(\mathbb{R}) \simeq \int_{\widehat{q\mathbb{Z}}}^\oplus \mathcal{H}_{\chi,q}^{\mathbb{R}}\, d\chi, \label{interHeisenbergdecomposion}
\end{equation}

where
\[
\mathcal{H}_{\chi,q}^{\mathbb{R}}
:=
\left\{ f\in L^2_{\mathrm{loc}}(\mathbb{R})
\,\middle|\, f(s+q)=\chi f(s)\ \text{for a.e. }s \right\},
\qquad
\|f\|_{\chi,q}^2
:=
\int_0^q |f(s)|^2\,ds.
\]
Thus these are Floquet fibers equipped with a fundamental-domain norm, not
closed subspaces of the usual global space $L^2(\mathbb{R})$.

To relate these spaces, we approximate $\mathcal{H}_{\chi,q}^{\mathbb{R}}$ by step functions. For each $m \in \mathbb{N}$, divide $[0,1)$ into $2^m$ equal subintervals $I_j = [j/2^m, (j+1)/2^m)$ and define

\[
\mathcal{H}_{m,\chi,q}^{\mathbb{R}} := \left\{ f \in \mathcal{H}_{\chi,q}^{\mathbb{R}} \,\middle|\, f \text{ is locally constant on each } I_j + r, \; r = 0, \ldots, q-1 \right\}.
\]

Define
\[
f_m(s)
:=
\lfloor s\rfloor
+2^{-m}\Bigl\lfloor 2^m\bigl(s-\lfloor s\rfloor\bigr)\Bigr\rfloor .
\]
Then $f_m(s+r)=f_m(s)+r$ for $r\in\mathbb Z$ and
$|f_m(s)-s|\leq2^{-m}$.  Define an approximate action $\rho_x^m(n)$ of
$\Gamma$ on $\mathcal{H}_{\chi,q}^{\mathbb{R}}$ by
\[
(\rho_x^m(n)\varphi)(s)
=
\exp\left(
\frac{2\pi i}{q}
(n_3 x_3+n_2 x_2+n_1 q x_1+f_m(s)n_2 q x_1)
\right)
\varphi(s+n_3).
\]
The subspace $\mathcal{H}_{m,\chi,q}^{\mathbb{R}}$ is invariant.  For
$j=0,\ldots,2^m-1$, its subspace
$\mathcal{H}_{m,\chi,q}^{I_j}$ supported on $I_j+\mathbb Z$ is invariant and
unitarily equivalent to
$\rho_{{\rm fin},(x_1,x_2+(j/2^m)q x_1,x_3)}$.  The union of the
step-function subspaces is dense in the fiber norm, and for every fixed
$n\in\Gamma$,
\[
\bigl\|\rho_x^m(n)-\rho_{L^2(\mathbb{R}),x}(n)\bigr\|
\leq
2\pi |n_2x_1|\,2^{-m}.
\]
Consequently the limiting action is taken in operator norm for each fixed
group element, while the finite direct sums over $j$ converge as Riemann
sums to the direct integral in~\eqref{discretetoHeisenbergLie}.
\end{proof}

\begin{remark}[Large-denominator interpretation]\label{remark34}
The exact generalized Pytlik inversion is supported on rational
finite-dimensional fibers, so no irrational representation is needed in the
rigorous lattice-side identity.  For a rational parameter \(h=p/q\), the
preceding formula is an exact
decomposition of the restricted Schr\"odinger representation into
\(q\)-dimensional Heisenberg fibers.  Thus, at rational flux, the
finite-dimensional matrices are not approximations to the Bloch fibers; they
are the Bloch fibers.

For an irrational parameter \(h\), there is no single finite denominator
\(q\), and hence no finite-dimensional decomposition of this exact form.
Rational numbers
\[
   h_\nu=\frac{p_\nu}{q_\nu}\longrightarrow h,
   \qquad q_\nu\longrightarrow\infty,
\]
provide a natural family of nearby parameters for which the exact rational
fiber decomposition is available.  The present paper uses this density only
to identify a large-denominator arithmetic skeleton.  It does not prove a
uniform operator-norm or smooth graph-norm convergence theorem comparing the
rational fibers with the irrational Schr\"odinger model.  Any such
quantitative comparison must specify the identifications, the observables,
and the topology, and is deferred to the companion analytic work.

The general nilpotent construction in
Section~\ref{Representationsdiscretenilpotent} has the same qualitative
feature: rational polarization characters are dense in the corresponding
compact abelian parameter spaces and yield exact finite-dimensional rational
fibers.  No stronger analytic convergence statement is asserted here.
\end{remark}

\section[Representations of nilpotent groups]
{Representations of Discrete Nilpotent Groups and Their Lie Completions}
\label{Representationsdiscretenilpotent}\label{sec:nilpotent}


This section contains the representation--theoretic core of the paper.
Our aim is to relate unitary representations of a torsion--free finitely generated
nilpotent discrete group $\Gamma$
to those of its Malcev completion $G$,
that is, the unique connected, simply connected nilpotent Lie group
in which $\Gamma$ embeds as a lattice,
and to establish canonical restriction and decomposition formulas
that generalize the Heisenberg--group case.

The emphasis is on structural results of a foundational nature.
Applications to spectral and asymptotic problems are not pursued here
and will be treated in a subsequent paper.

\subsection{Outline of the reduction procedure}
\label{ReductionOutline}

We summarize the reduction procedure that connects
unitary representations of the lattice $\Gamma$
with those of its Malcev completion $G$.
Each step corresponds to a precise structural result
that will be established in detail in the subsequent sections.
We use the label N for these nilpotent steps in order to keep them visibly
separate from the direct Heisenberg steps H-1 and H-2 in
Section~\ref{relationdiscLie}.  The labels are only a guide to the proof:
standard inputs and the new assembly into a Floquet--Bloch mechanism are
identified explicitly at each stage.

\begin{description}

\item[Step N-1.]
(Section~\ref{Step0-1})
We recall the factor representation decomposition of the right regular
representation of $\Gamma$ through Thoma's character space.  This is not a
classification of the full irreducible dual; it is the non--type~I starting
point from which the trace side of the theory is organized.

\item[Step N-2.]
(Section~\ref{Step0-2})
The factor representations relevant for the Plancherel theory are monomial.
After choosing compatible extensions of central characters from $Z(\Gamma)$
to $Z(G)$, they are compared with monomial representations of the Malcev
completion.  The extension is canonical only after the rational lift of the
central character has been fixed; changing the lift by an integral character
does not change the restriction to the lattice.

\item[Step N-3.]
(Section~\ref{Step0-3})
Fujiwara's Plancherel theorem decomposes these monomial representations of
$G$ into ordinary direct integrals of Kirillov representations.

\item[Step N-4.]
(Section~\ref{Step0-4})
We restrict the Kirillov representations of $G$ back to $\Gamma$.  For
rational Kirillov parameters this gives an ordinary Hilbert direct integral
over a compact torus whose fibers are induced representations of the lattice.
This is the first branching theorem.

\item[Step N-5.]
(Section~\ref{Step0-5})
On the rational odd locus, the induced fibers admit a further
finite-dimensional refinement.  This step is arithmetic and should be kept
separate from the first branching theorem.

\item[Step N-6.]
(Section~\ref{Step0-7})
We recall the part of Howe's classification of finite-dimensional
representations of torsion-free nilpotent groups used in Step~N-5.  This is a
standard input; the present paper uses it to identify the rational lattice
fibers obtained from the first restriction theorem.

\item[Step N-7.]
(Section~\ref{Step0-6})
A Pytlik-type Fourier inversion formula is obtained from finite-dimensional
traces.  The measure in this formula is finitely additive and is constructed
from finite quotients; it is not the measure of an ordinary Hilbert direct
integral.

\end{description}

\subsection
[Factor representations and Plancherel decomposition]
{Factor representations and Plancherel decomposition of discrete nilpotent groups (Step N-1)}
\label{Step0-1}


We begin by recalling the representation--theoretic framework
for decomposing the right regular representation of a discrete nilpotent group
$\Gamma$ into factor representations.
For a countable discrete group, Thoma's theorem detects the type~I
property group-theoretically: it holds exactly for groups that are virtually
abelian.  Thus a nonabelian torsion--free finitely generated nilpotent group
is not of type~I.  Glimm's theorem then explains the descriptive-set-theoretic
consequence: the irreducible dual cannot be used as a standard Borel
classification space in the same way as the character torus in the abelian
case.  This fundamental difficulty necessitates an alternative approach.
The material recalled in this step is a standard input from the
factor-Plancherel theory of countable groups; the present paper uses it only
as the starting trace-theoretic layer for the nilpotent Floquet--Bloch
construction.

Following the works of Bekka~\cite{BekkaPlancherel}, Johnston~\cite{Johnston},
and related developments,
one replaces $\widehat{\Gamma}$
with Thoma's dual space $\mathrm{Ch}(\Gamma)$,
which is a standard Borel space parametrizing
quasi--equivalence classes of finite--type factor representations
via extremal traces.

Recall that a \emph{von Neumann algebra}
is a self--adjoint subalgebra of $\mathcal{L}(H)$,
closed in the weak operator topology.
Such an algebra is called a \emph{factor}
if its center consists only of scalar multiples of the identity.
A unitary representation $\pi$ of $\Gamma$
is called a \emph{factor representation}
if the von Neumann algebra $\pi(\Gamma)''$
generated by $\pi(\Gamma)$ is a factor.

By Thoma's theorem~\cite{ThomaUnitary},
a countable discrete group is of type~I
if and only if it contains an abelian subgroup of finite index.
A nonabelian torsion--free finitely generated nilpotent group is not virtually
abelian, so its representation theory necessarily involves non--type~I
phenomena.

Let $\mathrm{Tr}(\Gamma)$ denote the compact convex set
of normalized, conjugation--invariant,
positive--definite functions on $\Gamma$.
Its extreme points form Thoma's dual space $\mathrm{Ch}(\Gamma)$.
Each $t\in\mathrm{Ch}(\Gamma)$ gives rise to a GNS representation $\pi_t$,
and Bekka's Plancherel theorem~\cite{BekkaPlancherel}
asserts that the right regular representation $\rho_R$ of $\Gamma$
admits the canonical decomposition
\[
\rho_R
=
\int_{\mathrm{Ch}(\Gamma)}^\oplus
\pi_t \, d\mu(t),
\]
where $\mu$ is the Plancherel measure on $\mathrm{Ch}(\Gamma)$.

For $\mu$--almost every $t$,
the representation $\pi_t$ is a factor representation.
Although non--factor representations may appear on dense subsets of
$\mathrm{Ch}(\Gamma)$,
the above decomposition provides a robust framework
for harmonic analysis on $\Gamma$ beyond the type~I setting.

For the nilpotent Plancherel layer used here, Bekka's and Johnston's
results give a more concrete description.  The generic factor
representations which occur in the factor-Plancherel decomposition are
realized as representations induced from unitary characters of the center
$Z(\Gamma)$~\cite{BekkaPlancherel,Johnston}.  In particular, they are monomial
representations, in the usual sense of being induced from a one-dimensional
unitary character of a subgroup.  We will use this special central-induced
model below; we do not mean that every monomial representation is induced
from the center.

\subsection{Extension to the Malcev completion via monomial representations (Step N-2)}
\label{Step0-2}

We now explain how the factor representations
arising in the Plancherel decomposition of $\Gamma$
extend to unitary representations of the Malcev completion $G$.
This extension step is included with explicit credit: the observation that
representations of a discrete nilpotent lattice can be extended to the Malcev
completion by using the Baker--Campbell--Hausdorff formula was communicated
to the author by Takashi Iwamoto.  What is used here is this standard
extension mechanism together with the rational central-character data fixed in
Step N-1.

Let $G$ be the connected, simply connected nilpotent Lie group
whose Lie algebra integrates the rational structure of $\Gamma$,
so that $\Gamma$ embeds as a lattice in $G$.
The center $Z(\Gamma)$ embeds discretely into $Z(G)$.  A unitary character
of $Z(\Gamma)$ extends to a character of $Z(G)$ after choosing a real lift of
its logarithm; the lift is not unique, but two choices differing by an
integral character have the same restriction to $Z(\Gamma)$.

Accordingly, each factor representation $\pi_t$ of $\Gamma$,
being induced from a character $\chi_t^\Gamma$ of $Z(\Gamma)$,
can be compared, after such a choice of lift, with a monomial unitary
representation of $G$, defined by
\[
\widetilde{\pi}_t := \mathrm{Ind}_{Z(G)}^G(\chi_t^G).
\]
This representation acts on the Hilbert space
\[
\widetilde{H}_t
=
\Bigl\{
f:G\to\mathbb{C}
\ \big|\ 
f(zg)=\chi_t^G(z)f(g)\ \text{for }z\in Z(G),
\ \int_{Z(G)\backslash G}|f(g)|^2\,dg<\infty
\Bigr\},
\]
by right translation,
\[
(\widetilde{\pi}_t(g)f)(\gamma)=f(\gamma g),
\qquad g,\gamma\in G.
\]

By construction, the restriction of $\chi_t^G$ to $Z(\Gamma)$ is the original character $\chi_t^\Gamma$, and the central character of $\widetilde{\pi}_t|_\Gamma$ agrees with that of the factor representation $\pi_t$.  The comparison is therefore made through the corresponding monomial model.
At this point, the analysis is transferred from the discrete group $\Gamma$
to monomial representations of the nilpotent Lie group $G$.

The decomposition of such representations into irreducible components
is governed by the orbit method.
In particular, the Plancherel decomposition of monomial representations,
developed by Fujiwara,
provides the essential bridge between the representation theory of $\Gamma$
and that of $G$.
We turn to this result in the next step.

\subsection
[Monomial representations and Fujiwara's theorem]
{Decomposition of monomial representations via Fujiwara's Plancherel theorem (Step N-3)}
\label{Step0-3}\label{subsec:monomial}


As observed in Section~\ref{Step0-2},
each factor representation $\pi_t$ of $\Gamma$
extends to a monomial representation
\[
\widetilde{\pi}_t = \mathrm{Ind}_{Z(G)}^G(\chi_t^G)
\]
of the Malcev completion $G$,
induced from a unitary character $\chi_t^G$ of the center $Z(G)$.

In this step, we apply
Fujiwara's Plancherel theorem~\cite{FujiwaraPlancherel,FujiwaraNote}
to decompose $\widetilde{\pi}_t$
into a direct integral of irreducible unitary representations of $G$.
This result plays a central role in our approach
and will ultimately lead to the restriction and decomposition formulas
for representations of $\Gamma$.

\subsubsection*{Monomial Representations and Linear Functionals}

Let $\chi_t^G : Z(G)\to \mathbb{T}$ be a unitary character
and let $d\chi_t^G : \mathfrak{z}\to i\mathbb{R}$
denote its differential,
where $\mathfrak{z}=\mathrm{Lie}(Z(G))$.
Choose a linear functional
$f\in\mathfrak{g}^*$
extending $d\chi_t^G$,
namely satisfying $f|_{\mathfrak{z}}=d\chi_t^G$.
Consider the affine subspace
\[
\Lambda_f
:=
\{\ell\in\mathfrak{g}^*\mid \ell|_{\mathfrak{z}}=f|_{\mathfrak{z}}\}
=
f+\mathfrak{z}^\perp.
\]

The space $\Lambda_f$
parametrizes all coadjoint orbits
compatible with the central character $\chi_t^G$
and serves as the natural parameter domain
for the irreducible representations
appearing in the decomposition of $\widetilde{\pi}_t$.

\subsubsection*{Fujiwara's Plancherel Decomposition}

According to Fujiwara's Plancherel theorem
(Theor\`eme~2 in~\cite{FujiwaraPlancherel}),
the monomial representation $\widetilde{\pi}_t$
admits the decomposition
\[
\widetilde{\pi}_t
\simeq
\int_{\Lambda_f}^\oplus
\pi_\ell\,
d\mu(\ell),
\]
where $\pi_\ell$
denotes the irreducible unitary representation of $G$
associated with $\ell\in\mathfrak{g}^*$
via the orbit method,
and $\mu$ is a Plancherel measure supported on $\Lambda_f$.

\subsubsection*{Stratification by Orbit Dimension}

Since the dimension of coadjoint orbits in $\Lambda_f$
is not constant,
we stratify $\Lambda_f$ by orbit dimension.
For each even integer $2e$,
let $U_e\subset\Lambda_f$
denote the Zariski--open subset of all $\ell$
such that the coadjoint orbit $\mathcal{O}_\ell$
has dimension $2e$.
Then the above decomposition refines to
\[
\widetilde{\pi}_t
\simeq
\bigoplus_e
\int_{U_e}^\oplus
\pi_\ell\, d\mu_e(\ell),
\]
where $\mu_e$ is a smooth measure supported on $U_e$.

Each stratum $U_e$
can be described as a union of affine subspaces
of the form
\[
V
=
f+\sum_{j\in T}\mathbb{R}X_j^*,
\]
where $T$ consists of non--jump indices,
namely directions along which the orbit dimension remains locally constant; see Appendix Theorem~\ref{corwin316} for the non-jump index convention.

\subsubsection*{Conclusion}

Through Fujiwara's decomposition,
the monomial representation $\widetilde{\pi}_t$
is expressed as a direct integral
of irreducible representations of $G$
parameterized by $\Lambda_f$.
In the subsequent substeps,
we restrict these representations back to the discrete subgroup $\Gamma$
and impose rationality conditions,
leading to finite--dimensional components
and ultimately to the main decomposition theorem
for nilpotent discrete groups.

\subsection
[Restriction to lattices]
{Restriction of irreducible unitary representations to lattices (Step N-4)}


In this section and the following one, we establish a decomposition formula
for the restriction to a lattice $\Gamma$ of an irreducible unitary representation
$\pi_l$  of a connected, simply connected nilpotent Lie group $G$ associate to $l \in \mathfrak{g}^\ast$.
This result constitutes the analytic and representation--theoretic core of the paper.

Our task is to decompose the restricted representation $\pi_l|_{\Gamma}$
as a direct integral of smaller unitary representations of $\Gamma$.
Such restriction problems for nilpotent groups have been studied previously.
The prototype for the ideal-polarization case is the theorem of
Bekka--Driutti~\cite{BekkaDriutti}, and the author learned of this paper from
Ali Baklouti.  When the polarization is not an ideal, we use the standard
polarization-replacement argument from Kirillov's proof, as presented for
example in Corwin--Greenleaf~\cite{CorwinGreenleaf}.

While several structural ideas already appear in the case of the
Heisenberg group, our main contribution lies in extending these arguments
to arbitrary torsion-free nilpotent groups
and in formulating decomposition results suitable for further analysis,
including finite-dimensional refinements developed in
Step N-5 (Section~\ref{Step0-5}).

As a guiding model,
we begin by revisiting the Heisenberg lattice
$\Gamma=\mathrm{Heis}_3(\mathbb{Z})$
from the viewpoint of the general orbit method.

\subsubsection{The Heisenberg Group as a Model for the Restriction Problem}\label{Step0-4}

We first illustrate Theorem~\ref{sec3decompose}
in the case of the Heisenberg group,
recasting Example~1 of~\cite{BekkaDriutti}
in a form adapted to our general framework.

\begin{example}[Heisenberg group revisited]\label{BekkaHeisenberg1}
Let $\mathfrak{g}=\mathbb{R}X_1\oplus\mathbb{R}X_2\oplus\mathbb{R}X_3$
be the Heisenberg Lie algebra with $[X_3,X_2]=X_1$,
and let $\{X_1^*,X_2^*,X_3^*\}$ denote the dual basis of $\mathfrak{g}^*$.
The corresponding Lie group
$G=\mathrm{Heis}_3(\mathbb{R})$
may be identified with $\mathbb{R}^3$ endowed with the group law
\[
(x_1,x_2,x_3)\cdot(y_1,y_2,y_3)
=
(x_1+y_1+x_3y_2,\; x_2+y_2,\; x_3+y_3).
\]

Let
$\Gamma=(\mathbb{Z},\mathbb{Z},\mathbb{Z})$
be the standard lattice.
Then $\{X_1,X_2,X_3\}$ is a strong Malcev basis of $\mathfrak{g}$
adapted to $\Gamma$.

Let $\pi$ be a nontrivial irreducible unitary representation of $G$.
Then there exists $\alpha\in\mathbb{R}\setminus\{0\}$
such that
$\pi\simeq\mathrm{Ind}_M^G(\chi_l)$,
where $l=\alpha X_1^*\in\mathfrak{g}^*$,
$\mathfrak{m}=\mathbb{R}X_1\oplus\mathbb{R}X_2$
is a polarization of $l$,
and $M=\exp(\mathfrak{m})$.

In this case, $\mathfrak{m}$ is an ideal of $\mathfrak{g}$,
and one obtains the restriction formula
\begin{equation}
\pi|_{\Gamma}
\simeq
\int_{[0,1)}^\oplus
\mathrm{Ind}_{M\cap\Gamma}^{\Gamma}
\bigl(\chi_{\alpha X_1^*+t\alpha X_2^*}|_{M\cap\Gamma}\bigr)\,dt.
\label{BekkaHeisenberg}
\end{equation}

Let $H=M\cap\Gamma$.  For a character $\chi$ of a normal subgroup
$H$ of $\Gamma$, write
\[
   \operatorname{St}^{\Gamma}_{H}(\chi)
   =
   \{\gamma\in\Gamma\mid
      \chi(\gamma h\gamma^{-1})=\chi(h)\text{ for all }h\in H\}
\]
for its stabilizer.  If $l_t=\alpha X_1^*+t\alpha X_2^*$, then a direct
calculation gives
\[
\operatorname{St}^{\Gamma}_{M\cap\Gamma}(\chi_{l_t})
=
\begin{cases}
   (\mathbb Z,\mathbb Z,0), & \alpha\notin\mathbb Q,\\
   (\mathbb Z,\mathbb Z,q\mathbb Z), &
      \alpha=p/q\in\mathbb Q\setminus\{0\},\ (p,q)=1.
\end{cases}
\]
By Mackey's irreducibility criterion, the first case gives an irreducible
integrand in \eqref{BekkaHeisenberg}.  The rational case is different: the
stabilizer is larger than $M\cap\Gamma$, and the induced representation is
therefore not yet irreducible.  It is precisely in this rational case that
one obtains the further decomposition into finite-dimensional components,
recovering the Fourier--type decomposition described in
Theorem~\ref{discretetoLie}.
\end{example}

In this example, formula~\eqref{BekkaHeisenberg} is the part that belongs to
the existing restriction theory of Bekka--Driutti.  The further rational
finite-dimensional fiber description is the Heisenberg finite-dimensional
Bloch organization developed in Step H-1, and the direct slicing proof of the
restriction formula is Step H-2.  Thus the example is repeated here not to
claim novelty for the standard orbit-method restriction formula, but to show
how the present finite-dimensionalization sits on top of that formula.

This example serves as a blueprint
for the general restriction and decomposition results
developed below.

\subsubsection
[Branching theorem I: statement]
{Branching theorem of nilpotent Lie groups to their lattices~I:
statement and structure of the proof}
\label{Branchingtheorem1}

We now turn to the general branching problem.
Let $\Gamma$ be a torsion--free, finitely generated nilpotent discrete group,
and let $G$ denote its Malcev completion.
Our objective is to generalize the decomposition
formula~\eqref{BekkaHeisenberg}
from the Heisenberg group
to arbitrary nilpotent groups.

More precisely, for an irreducible unitary representation
$\pi_l$ of $G$ associated with a rational functional
$l\in\mathfrak{g}_{\mathbb{Q}}^*$,
we seek an explicit and canonical decomposition
of the restricted representation
$\pi_l|_{\Gamma}$
as a direct integral of induced representations
of the lattice $\Gamma$.

Further refinement into finite--dimensional
irreducible representations,
analogous to~\eqref{interHeisenbergdecomposion},
will be carried out in
Step N-5 (Section~\ref{Step0-5}).
The present subsection focuses on the first
and structurally fundamental stage of this analysis.

\medskip
\paragraph{Rational structure.}
Let $G$ be a connected, simply connected nilpotent Lie group
admitting a lattice $\Gamma$.
Then $G$ possesses a rational structure in the sense of
Corwin~\cite{CorwinGreenleaf}:
there exists a basis
$\{X_1,\ldots,X_n\}$
of $\mathfrak{g}=\mathrm{Lie}(G)$
such that all structure constants
\[
[X_i,X_j] = \sum_{k=1}^n c_{ij}^k X_k
\qquad (i,j=1,\ldots,n)
\]
are rational.
Conversely, any such rational Lie algebra integrates to
a Lie group admitting a uniform lattice.

We denote by $\mathfrak{g}_{\mathbb{Q}}$
the rational structure determined by a lattice-adapted Malcev basis.
Although this is often described informally using $\log\Gamma$, the set
$\log\Gamma$ need not itself be an additive lattice in arbitrary
coordinates; the statements below are made in the Malcev-basis rational
structure.
For $l\in\mathfrak{g}^*$,
let $\mathfrak{r}_l$ denote the radical,
and let $\mathfrak{m}$ be a polarization
(maximal subordinate subalgebra).
If $l\in\mathfrak{g}^*_{\mathbb{Q}}$,
both $\mathfrak{r}_l$ and $\mathfrak{m}$
may be chosen inside $\mathfrak{g}_{\mathbb{Q}}$.
Write
\[
R_l=\exp(\mathfrak{r}_l),
\qquad
M=\exp(\mathfrak{m}).
\]

We are now ready to state the first general
decomposition theorem.

\begin{theorem}[First branching decomposition]
\label{firstdecomposition}
Assume that $l\in\mathfrak{g}_{\mathbb{Q}}^*$.
Then there exists a rational polarization
$\mathfrak{m}$ of dimension $m$
and a weak Malcev basis
$\{X_1,\ldots,X_m,\ldots,X_n\}$
of $\mathfrak{g}_{\mathbb{Q}}$
based on $\Gamma$
and passing through $\mathfrak{m}$,
such that
\begin{equation}
\pi_l|_{\Gamma}
=
\int_{[0,1)^{n-m}}^{\oplus}
\mathrm{Ind}_{M\cap\Gamma}^{\Gamma}
\bigl(
\chi_{\,l_{t_{m+1},\ldots,t_n}}
\big|_{M\cap\Gamma}
\bigr)
\;
dt_{m+1}\cdots dt_n ,
\label{pildecomposition}
\end{equation}
where
\[
l_{t_{m+1},\ldots,t_n}
=
\mathrm{Ad}^*\!\bigl(\exp(t_{m+1}X_{m+1})\bigr)
\cdots
\mathrm{Ad}^*\!\bigl(\exp(t_nX_n)\bigr)\, l .
\]
\end{theorem}

\begin{remark}
Theorem~\ref{firstdecomposition}
should be read in direct relation with Theorem~1.3 of~\cite{BekkaDriutti}.
In that theorem the polarization is assumed to be an ideal, and one may
work with a strong Malcev basis.  In the ideal-polarization case the proof
essentially runs only through Case~1 below, and the choice of a strong
basis does not create additional difficulties.

In the present theorem we do not assume that the polarization
\(\mathfrak m\) is an ideal.  Consequently one cannot use an arbitrary
strong Malcev basis in the same way.  Instead one chooses a weak Malcev
basis adapted to the rational polarization and, when necessary, replaces
the polarization by an equivalent one during the induction.  This additional step is not part of the restriction theorem of
Bekka--Driutti itself; it is the same kind of structural argument that
appears in the proof of Kirillov's theorem for nilpotent Lie groups
(cf.~\cite{CorwinGreenleaf}).
\end{remark}

\medskip
\paragraph{Inductive framework.}
The proof proceeds by induction on
$\dim\mathfrak{g}$.
The central tool is the following lemma,
which allows the reduction of the problem
to a lower--dimensional nilpotent Lie algebra.

\begin{lemma}[Kirillov's Lemma]
{\rm (cf.~\cite[Lemma~1.1.12]{CorwinGreenleaf})}
\label{Kirillovlemma}
Let $\mathfrak{g}$ be a noncommutative nilpotent Lie algebra
with one--dimensional center
$\mathfrak{z}(\mathfrak{g})=\mathbb{R}Z$.
Then $\mathfrak{g}$ admits a vector space decomposition
\[
\mathfrak{g}
=
\mathbb{R}Z
\oplus
\mathbb{R}Y
\oplus
\mathbb{R}X
\oplus
\mathfrak{w}
=
\mathbb{R}X \oplus \mathfrak{g}_0 ,
\]
where
$[X,Y]=Z$,
and
$\mathfrak{g}_0
=
\mathbb{R}Y\oplus\mathbb{R}Z\oplus\mathfrak{w}$
is an ideal of codimension one.
\end{lemma}

\begin{remark}
\label{Kirillovremark}
If $G$ admits a rational structure,
the decomposition in Lemma~\ref{Kirillovlemma}
may be chosen entirely inside
$\mathfrak{g}_{\mathbb{Q}}$.
\end{remark}

We now reduce to the case in which the center is one-dimensional and
$l$ is nonzero on it.  If the center contains a rational one-dimensional
subspace $\mathfrak h_{\mathbb Q}$ on which $l$ vanishes, then
$\mathfrak h_{\mathbb Q}$ is central and we may pass to the quotient
$\mathfrak g_{\mathbb Q}/\mathfrak h_{\mathbb Q}$, together with the image of
the polarization and of $l$.  The induction hypothesis applied to this
quotient gives the desired statement for the original representation.  After
repeating this harmless quotienting, we may therefore assume that the center
is one-dimensional and that $l$ is nonzero on it.

Choose a rational central generator $Z\in\mathfrak g_{\mathbb Q}$ with
$l(Z)\neq0$.  Since $l$ is rational, we may rescale $Z$ inside the rational
structure and normalize
\[
   \mathfrak z(\mathfrak g)=\mathbb RZ,
   \qquad l(Z)=1.
\]
By Lemma~\ref{Kirillovlemma} and Remark~\ref{Kirillovremark}, the elements
$X,Y,Z$ and the ideal $\mathfrak g_0$ may be chosen over
$\mathfrak g_{\mathbb Q}$ with $[X,Y]=Z$.  Finally, replacing $Y$ by
$Y-l(Y)Z$ preserves the relation $[X,Y]=Z$ and gives the additional
normalization $l(Y)=0$.

With these normalizations fixed, the proof splits into two mutually
exclusive cases:

\begin{description}
\item[Case~1.]
The polarization $\mathfrak m$ is contained in the codimension-one ideal
$\mathfrak g_0$.  In this case the restriction problem reduces to the
corresponding problem for $\mathfrak g_0$, and the result follows by the
inductive hypothesis, after keeping track of the extra circle parameter
arising from the quotient $G/G_0$.

\item[Case~2.]
The polarization $\mathfrak m$ is not contained in $\mathfrak g_0$.  In
this case one constructs another polarization
$\mathfrak m_1\subset\mathfrak g_0$ such that
$\pi_{l,\mathfrak m}\simeq\pi_{l,\mathfrak m_1}$.  This reduces the proof
back to Case~1, but the construction has to be made explicitly so that it
is compatible with the later finite-dimensional refinement.
\end{description}

The detailed proofs of these two cases are carried out in the next two
subsections.

\subsubsection
[Branching theorem II: induction step]
{Branching theorem of nilpotent Lie groups to their lattices~II:
Case~1 (the induction step)}
\label{inductionstep}


We first consider Case~1 in the proof of
Theorem~\ref{firstdecomposition},
namely the situation where the polarization
$\mathfrak{m}$ is contained in the codimension--one ideal
$\mathfrak{g}_0$ appearing in Lemma~\ref{Kirillovlemma}.

Although the argument closely follows that of
Theorem~1.3 in~\cite{BekkaDriutti},
we reproduce it here with modifications
adapted to the present setting,
using a weak Malcev basis
and making explicit the steps required
for later refinements.

\begin{proof}[Proof of Case~1]
Recall the setup.
Assume $\mathfrak{m}\subset\mathfrak{g}_0$,
and let $l_0:=l|_{\mathfrak{g}_0}$.
Set $G_0=\exp(\mathfrak{g}_0)$.
By the inductive hypothesis applied to $G_0$,
it suffices to show that
\[
\pi_{l,\mathfrak{m}}
\simeq
\mathrm{Ind}_{G_0}^G(\pi_{l_0,\mathfrak{m}}).
\]

We realize $\pi_{l,\mathfrak{m}}$
on the Hilbert space
\[
\mathcal{H}
=
\left\{
f:\mathbb{R}\to\mathcal{H}(\pi_{l_0,\mathfrak{m}})
\;\middle|\;
\int_{\mathbb{R}}\|f(t)\|^2\,dt<\infty
\right\}.
\]
For $g=g_0\exp(aX)$ with $g_0\in G_0$,
$a\in\mathbb{R}$,
the action is given by
\[
(\pi_{l,\mathfrak{m}}(g)f)(t)
=
\pi_{l_0,\mathfrak{m}}
\bigl(
\exp(tX)\,g_0\,\exp(-tX)
\bigr)
\,f(t+a).
\]

Fix $s\in[0,1)$
and define the Hilbert space
\[
\mathcal{H}_s
=
\left\{
f:\mathbb{Z}+s\to\mathcal{H}(\pi_{l_0,\mathfrak{m}})
\;\middle|\;
\sum_{p\in\mathbb{Z}}\|f(p+s)\|^2<\infty
\right\}.
\]
The restriction of $\pi_{l,\mathfrak{m}}$
to $\Gamma$
induces a representation $\rho_s$ on $\mathcal{H}_s$,
defined by
\[
(\rho_s(\gamma)f)(p+s)
=
\pi_{l_0,\mathfrak{m}}
\bigl(
\exp((p+s)X)\,\gamma_0\,\exp(-(p+s)X)
\bigr)
\,f(p+q+s),
\]
for $\gamma=\gamma_0\exp(qX)\in\Gamma$,
with $\gamma_0\in G_0\cap\Gamma$
and $q\in\mathbb{Z}$.

Define the unitary operator
\[
T:\mathcal{H}\to
\int_{[0,1)}^\oplus\mathcal{H}_s\,ds,
\qquad
(Tf)_s=f|_{\mathbb{Z}+s}.
\]
Then $T$ intertwines
$\pi_{l,\mathfrak{m}}|_\Gamma$
with $\int_{[0,1)}^\oplus\rho_s\,ds$.

Identifying each $\mathcal{H}_s$
with $\mathcal{H}_0
=
\ell^2(\mathbb{Z},\mathcal{H}(\pi_{l_0,\mathfrak{m}}))$,
we see that $\rho_s$
is induced from a representation
$\pi_0^s$ of $G_0$.
By the inductive hypothesis,
$\pi_0^s|_{G_0\cap\Gamma}$
admits a decomposition of the desired form.
Integrating over $s$
yields the decomposition
\eqref{pildecomposition},
completing the proof in Case~1.
\end{proof}

\subsubsection
[Branching theorem III: polarization]
{Branching theorem of nilpotent Lie groups to their lattices~III:
Case~2 (choice of polarization)}
\label{choicepolarization}


We next treat Case~2 of the proof of
Theorem~\ref{firstdecomposition},
namely the case where the polarization
$\mathfrak{m}$ is not contained in $\mathfrak{g}_0$.
The strategy is to replace $\mathfrak{m}$
by another polarization $\mathfrak{m}_1\subset\mathfrak{g}_0$
without changing the isomorphism class
of the induced representation.  This is the standard polarization-replacement
argument used in the proof of Appendix Theorem~\ref{nilpLieirredrep} \textup{(2)};
we recall the details here in the form needed for the restriction formula.

\begin{proof}[Proof of Case~2]
Set \(
   \mathfrak m_0=\mathfrak m\cap\mathfrak g_0.
\)
Since \(\mathfrak m\not\subset\mathfrak g_0\), the subspace
\(\mathfrak m_0\) has codimension one in \(\mathfrak m\).  We write
\[
   \mathfrak m=\mathfrak m_0\oplus\mathbb R X,
\]
after replacing \(X\) by \(X-X_0\) for a suitable \(X_0\in\mathfrak g_0\)
if necessary.  This replacement does not change the conclusion of
Kirillov's lemma.  Replacing \(X\) further by \(X-l(X)Z\), we may also
assume \(l(X)=0\).

Define
\[
   \mathfrak m_1=\mathfrak m_0\oplus\mathbb RY .
\]
Then \(Y\notin\mathfrak m_0\); otherwise \(X,Y\in\mathfrak m\) and
\(l([X,Y])=l(Z)=1\), contradicting the subordination of \(\mathfrak m\).
Thus \(\dim\mathfrak m_1=\dim\mathfrak m\).  Since \(Y\) is central in
\(\mathfrak g_0\) and \(l([\mathfrak m_0,Y])=0\), the subalgebra
\(\mathfrak m_1\) is subordinate to \(l\), and hence is another
polarization.  Moreover \(\mathfrak m_1\subset\mathfrak g_0\).

Let
\[
   \mathfrak k_0=\mathfrak m_0\cap\ker l,
   \qquad
   \mathfrak k=\mathfrak m_0\oplus\mathbb RX\oplus\mathbb RY\oplus\mathbb RZ,
   \qquad
   K=\exp\mathfrak k.
\]
A direct verification, using the subordination of \(\mathfrak m\) and the
nilpotence of \(\mathfrak g\), shows that \(\mathfrak k\) is a Lie
subalgebra and that \(\mathfrak k_0\) is an ideal of \(\mathfrak k\).
By induction in stages it is enough to prove
\[
   \mathrm{Ind}_{M}^{K}(\chi_l)
   \simeq
   \mathrm{Ind}_{M_1}^{K}(\chi_l),
\]
where \(M=\exp\mathfrak m\) and \(M_1=\exp\mathfrak m_1\).

Both induced representations have \(K_0=\exp\mathfrak k_0\) in their
kernels.  Indeed, for example, realizing
\(\mathrm{Ind}_M^K(\chi_l)\) on \(L^2(\mathbb R)\) using
\(\exp(\mathbb RY)\) as a transversal for \(M\backslash K\), an element of
\(K_0\) acts by the character \(\chi_l\) evaluated on an element of
\(\mathfrak k_0\subset\ker l\), hence acts trivially.  The same argument
applies to \(\mathrm{Ind}_{M_1}^K(\chi_l)\).

It remains to compare the quotient representations on \(K_0\backslash K\).
This is not merely a shorthand for ``modulo the kernel of \(l\)'': the relevant
quotient is the explicitly constructed Heisenberg quotient obtained by
killing the ideal \(\mathfrak k_0\subset\ker l\) inside the auxiliary algebra
\(\mathfrak k\).  Since \(\mathfrak m_0=\mathfrak k_0\oplus\mathbb RZ\), we have
\[
\begin{aligned}
   \mathfrak k
   &=\mathfrak k_0\oplus\mathbb RZ\oplus\mathbb RY\oplus\mathbb RX,\\
   \mathfrak m
   &=\mathfrak k_0\oplus\mathbb RZ\oplus\mathbb RX,\\
   \mathfrak m_1
   &=\mathfrak k_0\oplus\mathbb RZ\oplus\mathbb RY .
\end{aligned}
\]
Thus \(\mathfrak k/\mathfrak k_0\) is the Heisenberg Lie algebra with
central direction \(Z\), and the two quotient polarizations are
\(\mathbb RZ\oplus\mathbb RX\) and \(\mathbb RZ\oplus\mathbb RY\).

In the corresponding Schr\"odinger models on \(L^2(\mathbb R)\), the
quotient representations are given by
\[
\begin{aligned}
\bar\tau_1(\exp zZ)f(t)&=e^{2\pi iz}f(t),
&\bar\tau(\exp zZ)f(t)&=e^{2\pi iz}f(t),\\
\bar\tau_1(\exp yY)f(t)&=e^{2\pi ity}f(t),
&\bar\tau(\exp yY)f(t)&=f(t+y),\\
\bar\tau_1(\exp xX)f(t)&=f(t+x),
&\bar\tau(\exp xX)f(t)&=e^{-2\pi itx}f(t).
\end{aligned}
\]
The Euclidean Fourier transform
\[
   \mathcal Ff(t)=\int_{\mathbb R}e^{-2\pi itx}f(x)\,dx
\]
exchanges translation and modulation, and therefore intertwines
\(\bar\tau\) and \(\bar\tau_1\).  Hence
\(\mathrm{Ind}_{M}^{K}(\chi_l)\simeq
\mathrm{Ind}_{M_1}^{K}(\chi_l)\), and induction in stages gives
\(\pi_{l,\mathfrak m}\simeq\pi_{l,\mathfrak m_1}\).
Replacing \(\mathfrak m\) by \(\mathfrak m_1\subset\mathfrak g_0\) reduces
Case~2 to Case~1 and completes the proof of Theorem~\ref{firstdecomposition}.
\end{proof}

\subsection
[Finite-dimensional refinement and Fourier inversion]
{Finite--dimensional refinement and Fourier inversion (Steps N-5--N-7)}
\label{PartIV}

This final part completes the representation--theoretic framework
developed in the previous sections.
We first refine the decomposition obtained in
Theorem~\ref{firstdecomposition}
to finite--dimensional irreducible components
on a dense rational parameter set.
We then summarize the classification of such representations
following Howe,
and conclude with a Fourier inversion formula
generalizing Pytlik's theorem
to arbitrary torsion--free nilpotent discrete groups.

\subsubsection{Further decomposition into finite--dimensional representations (Step N-5)}
\label{Step0-5}

We return to the decomposition
\eqref{pildecomposition},
\[
\pi_l|_{\Gamma}
\simeq
\int_{[0,1)^{n-m}}^\oplus
\mathrm{Ind}_{M\cap\Gamma}^{\Gamma}
\bigl(\chi_{l_t}|_{M\cap\Gamma}\bigr)\,dt,
\]
where $l\in\mathfrak{g}_{\mathbb{Q}}^*$
and $\mathfrak{m}$ is a rational polarization.
The representations appearing as integrands
are in general not irreducible.
We now describe the further decomposition
into finite--dimensional irreducible representations
under an additional rationality condition
on the parameters.

Assume that for $t=(t_{m+1},\ldots,t_n)$,
each component belongs to $[0,1)\cap\mathbb{Q}^{\mathrm{odd}}$,
that is, each coordinate is a rational number with odd denominator, and
assume the corresponding odd-period condition in Howe's classification.
This is the clean odd-period case; Howe's general theorem also describes the
even-period case with a finite power-of-two ambiguity.

For such a parameter, use the invariant notation
$M_t\cap\Gamma$ and $\chi_t$ for the transported inducing pair (in the fixed
coordinate model, this is the pair represented by
$M\cap\Gamma$ and $\chi_{l_t}|_{M\cap\Gamma}$).  Let $P_t$ be a discrete
polarization for the transported lattice functional.  In lattice-adapted
Malcev coordinates one has
\[
M_t\cap\Gamma\triangleleft P_t,
\qquad
A_t:=P_t/(M_t\cap\Gamma)\ \text{is free abelian of finite rank},
\qquad
[\Gamma:P_t]<\infty.
\]
The character $\chi_t$ on $M_t\cap\Gamma$ extends to a character
$\chi_t^P$ on $P_t$.  All extensions are obtained by tensoring one fixed
extension with a character $\eta\in\widehat{A_t}$.  Induction in stages and
the ordinary Fourier transform on the abelian quotient give
\begin{align}
\mathrm{Ind}_{M_t\cap\Gamma}^{\Gamma}(\chi_t)
&\simeq
\mathrm{Ind}_{P_t}^{\Gamma}
\left(\mathrm{Ind}_{M_t\cap\Gamma}^{P_t}(\chi_t)\right)
\notag\\
&\simeq
\mathrm{Ind}_{P_t}^{\Gamma}
\left(\chi_t^P\otimes
\mathrm{Ind}_{M_t\cap\Gamma}^{P_t}(\mathbf 1)\right)
\notag\\
&\simeq
\int_{\widehat{A_t}}^\oplus
\mathrm{Ind}_{P_t}^{\Gamma}(\chi_t^P\otimes\eta)\,d\eta,
\label{intermidietedecomposition}
\end{align}
where $d\eta$ is Haar probability measure.  Howe's odd-period theorem
applies to every representation in the final integral, so each is
irreducible; it is finite-dimensional because $P_t$ has finite index in
$\Gamma$.  This is the parameter-dependent finite-dimensional refinement
used below.

\subsubsection
[Finite-dimensional representations and Howe's classification]
{Finite--dimensional irreducible representations and Howe's classification (Step N-6)}
\label{Step0-7}

Finite--dimensional irreducible unitary representations
of discrete, finitely generated, torsion--free nilpotent groups
were classified by Howe~\cite{Howe}.  We recall the part of the statement
needed here, including the elementary exponentiable convention used by
Howe.

Strictly speaking, for an arbitrary lattice \(\Gamma\) the set
\(\log\Gamma\subset\mathfrak g\) need not be closed under addition or under
the Lie bracket.  Howe avoids this nuisance by working first with
elementarily exponentiable, or e.e., lattices.  If \(G\) is \(k\)-step
nilpotent and \(L\subset\mathfrak g\) is a lattice satisfying
\([L,L]\subset k!L\), then the Baker--Campbell--Hausdorff formula implies
that \(\Gamma=\exp L\) is a subgroup of \(G\), and \(L=\log\Gamma\) may be
used as a Lie lattice.  By Howe's Proposition~0, an arbitrary
finitely generated torsion--free nilpotent group may be reduced, up to
finite index, to this e.e. situation.  Thus the notation \(L=\log\Gamma\)
in the following theorem is to be understood in this e.e. sense.

Let \(\Gamma\) be e.e. and put \(L=\log\Gamma\).  If
\(\psi\in\widehat L\) is rational on the derived lattice, define
\[
R_\psi
=
\{X\in L\mid \psi([X,Y])=0\text{ for all }Y\in L\}.
\]
Then \(R_\psi\) has finite index in \(L\), and one chooses a polarizing
subalgebra \(P\supset R_\psi\) so that \(P/R_\psi\) is a maximal isotropic
subgroup for the induced alternating form.

\begin{theorem}[Theorem 1 in~\cite{Howe}]
\label{Howefinite}
Let \(\Gamma\) be an e.e. discrete, finitely generated, torsion-free
nilpotent group.  Put
\[
   L=\log\Gamma,
   \qquad L'=2L,
   \qquad \Gamma'=\exp L',
   \qquad L_s^{(2)}=\log\Gamma_s^{(2)},
\]
where \(\Gamma_s^{(2)}\) denotes the saturation of the commutator subgroup,
and put \(n\Gamma=\exp(nL)\) for \(n>2\).
\begin{description}
\item[\textup{(a)}]
Let \(O\) be a finite \(\operatorname{Ad}^*\Gamma\)-orbit in \(\widehat L\),
and let \(\psi\in O\).  Let \(n\) be the period of
\(\psi|_{L_s^{(2)}}\); this is called the period of \(O\).  If \(n\) is odd,
then a finite-dimensional irreducible unitary representation is associated
with \(O\) as follows.  If \(P\) is an e.e. polarizing subalgebra for
\(\psi\), \(\Pi=\exp P\), and \(\widetilde\psi=\psi|_{P}\), then
\[
   U^{\widetilde\psi}
   =
   \operatorname{Ind}_{\Pi}^{\Gamma}(\chi_{\widetilde\psi})
\]
is finite-dimensional and irreducible.  Its dimension is \(|O|^{1/2}\), and
its character is
\[
   |O|^{-1/2}\sum_{\varphi\in O}\varphi .
\]
All representations obtained from such odd-period orbits, up to twisting by
characters of \(\Gamma/\Gamma_s^{(2)}\), are realized in this manner.

\item[\textup{(b)}]
In general, there is a surjective map from the finite-dimensional irreducible
unitary representations of \(\Gamma'\) to the finite
\(\operatorname{Ad}^*\Gamma\)-orbits in \(\widehat{L'}\).  This map is at most
\(|\Gamma/\Gamma'|\)-to-one.  If \(\{U_i\}_{i=1}^\ell\) maps onto the orbit
\(O\), then the representations \(U_i\) are permuted transitively by
\(\operatorname{Ad}^*\Gamma/\Gamma'\), so that they define a point in
\(M(\Gamma',\Gamma)\).  The integer \(\ell\) is a power of two; if \(m\) is
the common dimension of the \(U_i\), then
\[
   \ell m^2=|O|.
\]
If \(\xi_i\) are the characters of \(U_i\), then
\[
   \sum_i \xi_i
   =
   m^{-1}\sum_{\varphi\in O}\varphi .
\]
\end{description}
\end{theorem}

The odd-period hypothesis in Step~N-5 is therefore only a technical
convenience.  It puts us in the simpler case described in part~(a), where a
single induced representation already gives the desired irreducible
finite-dimensional representation and the character formula has no power-of-two
ambiguity.  Part~(b) shows that the general rational case is still covered by
Howe's classification; one only has to keep track of a finite, explicitly
controlled ambiguity.  Thus the odd denominator condition is not a conceptual
restriction, but a way of keeping the finite-dimensional refinement explicit.

\subsubsection{Main decomposition theorem for rational parameters}

We may now summarize the structural part of the paper.  The statement is
split into two assertions in order to avoid the misleading impression that
the first Hilbert direct integral is already a finite-dimensional
Plancherel decomposition.

\begin{theorem}[Restriction and finite-dimensional rational refinement]
\label{sec3decompose}
Let $\Gamma$ be a torsion--free finitely generated nilpotent group,
let $G$ be its Malcev completion, and let $l\in\mathfrak{g}_{\mathbb{Q}}^*$.
Then there exist rational subordinate data and a lattice-adapted weak
Malcev system such that
\[
\pi_l|_{\Gamma}
\simeq
\int_{[0,1)^{n-m}}^\oplus
\mathrm{Ind}_{M_t\cap\Gamma}^{\Gamma}
(\chi_t)\,dt.
\]
In a fixed coordinate model this may be written as
\[
\pi_l|_{\Gamma}
\simeq
\int_{[0,1)^{n-m}}^\oplus
\mathrm{Ind}_{M\cap\Gamma}^{\Gamma}
\bigl(\chi_{l_t}|_{M\cap\Gamma}\bigr)\,dt,
\]
with the understanding that, in the non-ideal polarization case, the
inducing pair is transported through the standard polarization-replacement
intertwiner.

If in addition the corresponding lattice functional lies on the rational
odd locus, then the induced fiber admits the further refinement
\[
\mathrm{Ind}_{M_t\cap\Gamma}^{\Gamma}(\chi_t)
\simeq
\int_{\widehat{P_t/(M_t\cap\Gamma)}}^\oplus
\mathrm{Ind}_{P_t}^{\Gamma}(\chi_t^P\otimes\eta)\,d\eta,
\]
where $P_t$ is a discrete polarization, $\chi_t^P$ is an
extension of $\chi_t$ to $P_t$, and every representation in the final
integral is finite-dimensional and irreducible.
\end{theorem}

\begin{remark}
The first integral in Theorem~\ref{sec3decompose} is an ordinary Hilbert
direct integral over real Bloch parameters.  The final finite-dimensional
statement is an arithmetic refinement on the rational finite-dimensional locus.  It is
this second family, not the first direct integral by itself, that enters the
Pytlik-type finitely additive inversion formula.
\end{remark}

\subsubsection{A Fourier inversion formula for nilpotent lattices (Step N-7)}\label{Step0-6}

The restriction theorem and the finite-dimensional rational refinement explain which
finite-dimensional representations occur from rational Kirillov data.  The
route from these fibers to an inversion formula is indirect.  One cannot
simply apply the ordinary Kirillov--Plancherel formula on the Malcev
completion and then restrict the result to the lattice.  The point is already
visible in the abelian model \(G=\R\), \(\Gamma=\Z\): the character
\(x\mapsto \ee^{2\pi\ii\xi x}\) of \(\R\) restricts to a character of
\(\Z\) depending only on \(\xi\) modulo \(\Z\), and a nontrivial character of
\(G\) may become trivial on \(\Gamma\).  For nilpotent lattices the same
issue is encoded by rational central characters, polarizations, and
lattice-period data.

Thus one first passes through the lattice Plancherel theory: decompose the
regular representation of \(\Gamma\) into factor representations, extend
these factors to monomial representations of \(G\), apply Fujiwara's
Plancherel formula, and only then restrict the resulting Kirillov
representations back to \(\Gamma\) and apply the finite-dimensional
refinement.  The finitely additive Plancherel measure below is constructed
on the finite-dimensional rational fibers obtained at the end of this process.

The following inversion formula is obtained by replacing the ordinary
parameter spaces above with their rational finite-dimensional subsets and by using the
corresponding finitely additive volume convention, in the same spirit as
Pytlik's formula for the discrete Heisenberg group.  This separation is
important: the measure below is finitely additive, hence it is a trace
functional and not the measure of a Hilbert-space direct integral.

\begin{theorem}[Generalized Pytlik theorem]
\label{newPytlik}
Let $\Gamma$ be a finitely generated, torsion--free nilpotent group.  There
exists a positive finitely additive probability measure $\mu$ on the set
$\widehat\Gamma_{\mathrm{fin}}$ of finite-dimensional irreducible unitary
representations of $\Gamma$ such that, for every $f\in\ell^1(\Gamma)$ and
$\sigma\in\Gamma$,
\[
f(\sigma)
=
\int_{\widehat\Gamma_{\mathrm{fin}}}
\frac{1}{\dim\pi}\,
\mathrm{Tr}\bigl(\pi(\sigma^{-1})\pi(f)\bigr)\,d\mu(\pi).
\]
The corresponding normalized Plancherel identity is obtained by applying the
same formula to $f^**f$.
\end{theorem}

\begin{proof}[Construction of the finitely additive measure]
The construction is the natural nilpotent analogue of Pytlik's measure for
the discrete Heisenberg group.  At each stage of the preceding construction
there is an ambient compact abelian parameter space, or an affine
Kirillov--Bloch parameter torus, which we denote here by \(T\).  The
finite-dimensional representations of nilpotent lattices occur on a dense rational skeleton
\(D\subset T\).

The finitely additive content is defined in a cube-first form.  We do not
assign to an arbitrary subset of \(D\) the Lebesgue measure of its closure.
Instead, we first take a regular set \(B\subset T\), for instance a finite
union of half-open cubes, or more generally a Jordan measurable set with
Lebesgue-null boundary, and then consider the subset \(D\cap B\) of the
rational finite-dimensional skeleton.  We set
\[
   \mu_D(D\cap B)=\operatorname{Leb}_T(B).
\]
This gives a finitely additive content on the Boolean algebra generated by
such regular rational sets.  In the Heisenberg case this is precisely the
geometric Pytlik convention: one keeps only the rational finite-dimensional
fibres, but measures regular families of them by the Lebesgue content of the
corresponding region in the ambient Bloch torus.

Combining this finitely additive parameter content with the normalized
finite-dimensional traces of the finite-dimensional rational fibers gives the Pytlik-type trace
functional.  For group-algebra elements the Fourier inversion identity
follows from the ordinary Fourier inversion formulas on the abelian
parameter spaces and from the finite-dimensional character formula in Howe's
theorem.  The normalized trace factor \(\dim(\pi)^{-1}\) is the same one that
appears in Pytlik's formula for the discrete Heisenberg group.

For finitely supported functions this gives the geometric realization of
the displayed coefficient functional on the regular rational algebra
described above.  The existence of a positive finitely additive probability
measure and the extension to \(\ell^1(\Gamma)\) are supplied in
Appendix~\ref{finiteadditiveappendix} by a finite-quotient ultralimit.  On the
coefficient functions occurring in the inversion formula, the two
realizations give the same inversion values and normalized-trace convention.
We do not assert literal equality of the resulting finitely additive set
functions.  In particular, the finite-quotient realization may depend on the
chosen residual chain and free ultrafilter.
\end{proof}

This formula generalizes Pytlik's Fourier inversion for the discrete
Heisenberg group.

\begin{remark}[Large denominators in the nilpotent case]
The qualitative large-denominator interpretation of
Remark~\ref{remark34} is not special to the central parameter of
\(\Heis_3\).  In the nilpotent case, the relevant polarization characters
form compact abelian parameter spaces, and their rational points with
arbitrarily large denominators are dense.  The induced representations at
those rational points are the exact finite-dimensional lattice fibers used
above.  The present paper does not infer from density alone a quantitative
operator estimate after induction.  Any comparison with the Kirillov model
on the Malcev completion must specify the observables and topology and is
left to the companion analytic work.

The last qualification is only to avoid a possible misreading of the integral
notation.  Since the Pytlik-type functional in Theorem~\ref{newPytlik} is
finitely additive, it is already clear that it is not a countably additive
Borel Plancherel measure producing a Hilbert direct integral over the full
unitary dual.  We use it here as a finite-dimensional Fourier inversion and
normalized-trace identity.  For the Heisenberg lattice this recovers the
large-denominator rational central characters of Pytlik's formula; for a
general nilpotent lattice, the analogous role is played by rational
polarization characters and the induced finite-dimensional rational fibers.
\end{remark}

\section{Analytic and Geometric Structures Associated with Nilpotent Groups}

In this section, we introduce the geometric and analytic structures
that form the analytic counterpart
of the representation--theoretic framework developed earlier.
These constructions do not constitute additional steps
of the reduction procedure,
but rather provide a fixed setting
in which the results of the preceding sections
can be interpreted and applied.

We first recall filtrations and stratifications
of nilpotent Lie algebras,
leading to canonical stratified (graded) approximations
of the Lie group $G$.
This structure will be used throughout the subsequent analysis.
Based on this framework,
we associate to irreducible unitary representations of $G$
certain canonical hypo--elliptic differential operators,
which serve as analytic avatars
of the representation--theoretic objects appearing earlier.
Finally, these constructions allow us to relate
analytic problems on the lattice $\Gamma$
to corresponding questions on the Lie group $G$
in a precise and systematic manner.

\subsection{Filtrations of Lie Algebras and Stratified Lie Groups}
\label{stratification}

In this subsection, we recall the notions of filtrations and stratifications
of nilpotent Lie algebras, following \cite{Alexopoulos3}.
These structures play a central role in the construction
of associated graded Lie algebras
and their corresponding stratified Lie groups,
which serve as tangent cones
in the analysis of hypoelliptic operators.

\medskip

Let $\mathfrak{g}$ be the Lie algebra of a nilpotent Lie group $G$,
which we identify with the left-invariant vector fields on $G$.
We define the descending central series by
\[
\mathfrak{g}_1=\mathfrak{g}, \qquad
\mathfrak{g}_{i+1}=[\mathfrak{g}_1,\mathfrak{g}_i], \quad i\ge 1.
\]
Since $\mathfrak{g}$ is nilpotent, there exists an integer $n$ such that
\[
\mathfrak{g}
=
\mathfrak{g}_1
\supset \mathfrak{g}_2
\supset \cdots
\supset \mathfrak{g}_n
\supset \mathfrak{g}_{n+1}=\{0\},
\qquad
\mathfrak{g}_n\neq\{0\}.
\]

We choose linear subspaces
$\mathfrak{g}^{(1)},\ldots,\mathfrak{g}^{(n)}$
of $\mathfrak{g}$ such that
\begin{equation}
\mathfrak{g}_i
=
\mathfrak{g}^{(i)}\oplus \cdots \oplus \mathfrak{g}^{(n)},
\qquad 1\le i\le n.
\label{liealgdecomposition}
\end{equation}

Set
\begin{align*}
g_i
&:= \dim(\mathfrak{g}/\mathfrak{g}_{i+1})
   = \dim(\mathfrak{g}^{(1)}\oplus \cdots \oplus \mathfrak{g}^{(i)}),
   \qquad 1\le i\le n,\\
\sigma(j)
&:= i \quad \text{for } g_{i-1}<j\le g_i,\\
q &:= g_n = \dim \mathfrak{g}.
\end{align*}

The \emph{homogeneous dimension} of $G$ is then defined by
\[
D = \sigma(1)+\cdots+\sigma(q).
\]

\begin{definition}
The Lie algebra $\mathfrak{g}$ is called \emph{stratified}
if the following conditions are satisfied:
\begin{align}
\mathfrak{g}
&= \mathfrak{g}^{(1)}\oplus \cdots \oplus \mathfrak{g}^{(n)}, 
\label{liestratified1}\\
[\mathfrak{g}^{(i)},\mathfrak{g}^{(j)}]
&\subset \mathfrak{g}^{(i+j)},
\label{liestratified2}\\
\mathfrak{g}^{(1)}
&\text{generates }\mathfrak{g}\text{ as a Lie algebra}.
\label{liestratified3}
\end{align}
\end{definition}

Choose a basis $\{X_1,\ldots,X_q\}$ of $\mathfrak{g}$ such that
$\{X_{g_{i-1}+1},\ldots,X_{g_i}\}$
forms a basis of $\mathfrak{g}^{(i)}$ for $1\le i\le n$.

We define a new Lie bracket $[\cdot,\cdot]_0$ on $\mathfrak{g}$ by
\begin{equation}
[X_i,X_j]_0
=
\operatorname{pr}_{\mathfrak{g}^{\sigma(i)+\sigma(j)}}[X_i,X_j],
\label{zerobracket}
\end{equation}
where $\operatorname{pr}_{\mathfrak{g}^{\sigma(i)+\sigma(j)}}$
denotes the projection onto $\mathfrak{g}^{\sigma(i)+\sigma(j)}$.
We denote by $\mathfrak{g}_0=(\mathfrak{g},[\cdot,\cdot]_0)$
the resulting nilpotent Lie algebra.

Using exponential coordinates of the second kind (Malcev coordinates),
we identify $G$ with $\mathbb{R}^q$ via
\[
\phi:\mathbb{R}^q\to G,
\qquad
(x_q,\ldots,x_1)\mapsto
\exp(x_qX_q)\cdots\exp(x_1X_1).
\]

For $\varepsilon>0$, define dilations on $G$ by
\[
\delta_\varepsilon(x_q,\ldots,x_1)
=
(\varepsilon^{\sigma(q)}x_q,\ldots,\varepsilon^{\sigma(1)}x_1).
\]
Let $\ast_\varepsilon$ be the rescaled group law
\[
x\ast_\varepsilon y
=
\delta_\varepsilon\bigl(
(\delta_{\varepsilon^{-1}}x)(\delta_{\varepsilon^{-1}}y)
\bigr),
\]
and define
\[
x\ast_0 y
=
\lim_{\varepsilon\to 0} x\ast_\varepsilon y.
\]
Then $G_0=(G,\ast_0)$ is a stratified nilpotent Lie group
whose Lie algebra is canonically isomorphic to $\mathfrak{g}_0$.
We refer to $G_0$ as the \emph{nilpotent (or stratified) approximation} of $G$.

We identify $G$ and $G_0$ as smooth manifolds.
In canonical $\ast_0$-coordinates
\[
\phi^\ast:\mathbb{R}^q\to G_0,
\qquad
(x_{q,\ast},\ldots,x_{1,\ast})
\mapsto
\exp(x_{q,\ast}X_q)\ast\cdots\ast\exp(x_{1,\ast}X_1),
\]
the relations between the original coordinates and the stratified ones
are given inductively by
\begin{align}
x_{i,\ast} &= x_i,
&&1\le i\le \sigma(1)+\sigma(2),\notag\\
x_{i,\ast}
&= x_i + \sum_{0<|K|\le k} C^K P_K(x),
&&\sum_{j=1}^k\sigma(j)+1 \le i\le \sum_{j=1}^{k+1}\sigma(j),
\label{stratifiedcanonical}
\end{align}
where $K=(i_1,\ldots,i_\ell)$ is a multi-index,
$|K|=\sigma(i_1)+\cdots+\sigma(i_\ell)$,
and $P_K(x)=x_{i_1}\cdots x_{i_\ell}$.

\medskip

In the sequel, we freely use this stratified structure.
It is implicit in the proof of the restriction theorem
in Step N-4 (Subsection~\ref{Step0-4})
and will be used explicitly in the construction of the canonical
hypo-elliptic operator on $G$
in Subsection~\ref{constracthypoelliptic}.

\subsection[Concrete construction of the hypo-elliptic operator HG]{Concrete construction of the hypo-elliptic operator $\mathcal H^G$}
\label{constracthypoelliptic}

We record a concrete reference procedure for constructing the canonical
hypo-elliptic operator associated with the generic Kirillov representations of
a nilpotent Lie group.  The material in this subsection is standard in the
orbit method for nilpotent Lie groups; its purpose here is only to fix the
notations used later and to make clear how the operator \(\mathcal H^G\) is
obtained from the Lie-algebraic data of \(\Gamma\).

\begin{description}[leftmargin=2.5em]
\item[\rm (1)]
Let \(\Gamma\) be a torsion-free finitely generated nilpotent group and let
\(G\) be its Malcev completion.  The passage from \(\Gamma\) to \(G\) is
natural because, after choosing Malcev coordinates, the multiplication law in
\(\Gamma\) is given by polynomial functions.  Thus \(\Gamma\) is a lattice in
the connected simply connected nilpotent Lie group \(G\).

\item[\rm (2)]
Using the filtration and contraction procedure of
Subsection~\ref{stratification}, we replace \(G\), when needed, by its
stratified nilpotent approximation.  In the remainder of this construction we
assume that the model Lie group is stratified,
\[
  \mathfrak g = \mathfrak g^{(1)}\oplus \mathfrak g^{(2)}\oplus\cdots .
\]
For a general nilpotent group this should be read as the associated graded
model determined by the chosen filtration.

\item[\rm (3)]
Choose a rational strong Malcev basis
\(\{X_1,\ldots,X_n\}\) of \(\mathfrak g\), compatible with the stratification.
In particular, we choose \(X_1,\ldots,X_k\) to span the first layer
\(\mathfrak g^{(1)}\).  Such a rational choice is possible because the
existence of the lattice gives a \(\mathbb Q\)-structure on \(\mathfrak g\)
(see, for example, \cite[Theorem 5.1.8]{CorwinGreenleaf}).

\item[\rm (4)]
Compute the coadjoint action of the basis elements on
\(\mathfrak g^*_{\mathbb Q}\).  Choose a generic element
\(l\in\mathfrak g^*_{\mathbb Q}\), meaning that its stabilizer
\[
  \mathfrak g_l=\{X\in\mathfrak g\mid \operatorname{ad}^*(X)l=0\}
\]
has minimal dimension.  Equivalently, the coadjoint orbit \(\mathcal O_l\)
has maximal dimension.  The set of such generic elements is Zariski open.

\item[\rm (5)]
For such an \(l\), let \(\mathfrak r_l\) be the radical of the alternating
form \((X,Y)\mapsto l([X,Y])\).  A polarization is a maximal subordinate
subalgebra \(\mathfrak m\subset\mathfrak g\), and it must contain
\(\mathfrak r_l\).  In general there may be many polarizations.  The following
standard construction of Vergne provides a canonical choice once a strong
Malcev basis, or equivalently a chain of ideals, has been fixed; see
\cite[Theorem 1.3.6]{CorwinGreenleaf}.

\begin{theorem}[Vergne's construction]
\label{methodVergne}
Let
\[
  (0)\subset \mathfrak g_1\subset\mathfrak g_2\subset\cdots\subset
  \mathfrak g_n=\mathfrak g
\]
be a chain of ideals with \(\dim \mathfrak g_j=j\).  For
\(l\in\mathfrak g^*\), put \(l_j=l|_{\mathfrak g_j}\), and let
\(\mathfrak r_{l_j}\) be the radical of \(l_j([\cdot,\cdot])\) on
\(\mathfrak g_j\).  Then
\[
  \mathfrak m_l=\sum_{j=1}^n \mathfrak r_{l_j}
\]
is a polarizing subalgebra for \(l\).
\end{theorem}

For the strong Malcev basis above, the chain may be taken as
\[
\mathfrak g_j=\mathbb R\operatorname{-span}\{X_1,\ldots,X_j\}.
\]

\item[\rm (6)]
To describe the generic orbits, one may use the standard rational cross-section
in the sense of Corwin--Greenleaf.  With respect to the chosen Jordan--H\"older
basis, let \(d_j(l)\) denote the dimension of the orbit of the restriction of
\(l\) to \(\mathfrak g_j\), and let \(d(l)=(d_1(l),\ldots,d_n(l))\).  On a
Zariski open layer this multi-index is constant.  The jump indices determine
coordinates along the orbit and the non-jump indices give a transversal
\(V_T\) parametrizing the generic orbits.  This is the cross-section appearing
in the usual Kirillov--Fujiwara Plancherel formula; see also
Appendix Theorem~\ref{corwin316} for the non-jump index convention.

\item[\rm (7)]
Choose \(l\) in a generic orbit and a polarization \(\mathfrak m=\mathfrak
m_l\), for instance the Vergne polarization above.  With \(M=\exp\mathfrak m\),
define the character
\[
  \chi_l(\exp X)=e^{2\pi i l(X)},\qquad X\in\mathfrak m,
\]
and realize the associated irreducible unitary representation by
\[
  \pi_l=\operatorname{Ind}_{M}^{G}(\chi_l).
\]

\item[\rm (8)]
Let \(\omega_1,\ldots,\omega_k\) be the harmonic one-forms or Abel--Jacobi data
attached to the first-layer generators \(X_1,\ldots,X_k\).  The model
hypo-elliptic operator associated with the generic representation \(\pi_l\) is
then
\begin{equation}
\mathcal H^G
=
-\sum_{j=1}^{k}
\|\omega_j\|_{L^2(M)}^2\,d\pi_l(X_j)^2 .
\label{detailedgeneralhypo1}
\end{equation}
Equivalently, since each \(d\pi_l(X_j)\) is skew-adjoint on smooth vectors, this
is the positive sum of squares determined by the first layer.  If one uses the
opposite sign convention for the Laplacian, the displayed sign is absorbed into
that convention.  This is the operator denoted \(\mathcal H^G\) in the analytic
applications.
\end{description}

\begin{remark}
The construction above is not meant as a new contribution to the orbit method.
It is included only as a compact recipe for readers who want to pass from a
nilpotent lattice \(\Gamma\) to the concrete polynomial Schr\"odinger operator
which appears after choosing a generic Kirillov model.  In explicit examples,
such as the Heisenberg and Engel groups, formula~\eqref{detailedgeneralhypo1}
reduces to the familiar harmonic or higher polynomial oscillator.
\end{remark}

\section
[The Harper operator in the foundational theory]
{The Harper Operator in the Foundational Theory:
Exact Fibers and No Semiclassical Duplication}
\label{Harperfoundation}\label{sec:harper}

The Harper operator is included here as a model showing where the
finite-dimensional representations constructed above appear concretely.  The
detailed semiclassical expansion, the motion of rational bands, and the
comparison with Landau levels are part of the analytic companion.  To avoid
duplication, we record only the exact representation-theoretic input needed
from the foundations.

Let
\[
H_\theta:\ell^2(\mathbb Z^2)\longrightarrow \ell^2(\mathbb Z^2)
\]
be the magnetic Harper operator
\[
(H_\theta u)(m,n)
=
u(m+1,n)+u(m-1,n)
+e^{im\theta}u(m,n+1)+e^{-im\theta}u(m,n-1),
\]
where the four summands are the basic magnetic translations.  Equivalently,
after choosing generators $u,v$ of the discrete Heisenberg group
$\Gamma=\mathrm{Heis}_3(\mathbb Z)$ with central commutator $w=[u,v]$, the
operator is the image of
\[
A_{\mathrm{Har}}=u+u^{-1}+v+v^{-1}
\]
in a representation with central character $w\mapsto e^{i\theta}$.

If $\theta=2\pi p/q$ with $(p,q)=1$, the relevant Heisenberg
representations are $q$--dimensional.  The exact finite-dimensional Bloch
fibers are
\[
H_x
=
\rho_{{\rm fin},x}(u)+\rho_{{\rm fin},x}(u)^*
+\rho_{{\rm fin},x}(v)+\rho_{{\rm fin},x}(v)^*,
\qquad x=(p/q,x_2,x_3),
\]
and the rational magnetic band decomposition is
\[
\sigma(H_\theta)
=
\bigcup_{x_2,x_3\in[0,1]}
\sigma(H_x).
\]
This is the usual rational Harper decomposition, but in the present paper it
is interpreted as the Heisenberg instance of the finite-dimensional rational
refinement obtained from the restriction formula and the finite-dimensional
classification.  No limiting or perturbative argument is involved in this
identity.  Only these rational finite-dimensional fibers enter the exact
generalized Pytlik inversion on the lattice side.  Irrational
Schr\"odinger representations are not part of its support; below they are
used only as smooth ambient models for organizing coefficient calculations.

The same central parameter also has a Schr\"odinger realization on
$L^2(\mathbb R)$.  The classical Wilkinson expansion was first obtained by
Wilkinson~\cite{Wilkinson} by WKB arguments, and related derivations were
given by Rammal and Bellissard~\cite{RammalBellissard}.  In the usual semiclassical normalization, the Harper operator is
unitarily equivalent near the bottom of the spectrum to
\begin{equation}
   h_\theta
   =
   -2\cos\left(\sqrt{\theta}\,\frac{1}{\mathrm i}\frac{d}{ds}\right)
   -2\cos(\sqrt{\theta}\,s),
\label{foundation-harper-schrodinger}
\end{equation}
up to harmless conventions for the flux and the coordinate $s$.  This
expression is closely related to the Schr\"odinger representation of
$\mathrm{Heis}_3(\mathbb R)$.  The formal Taylor expansion gives
\begin{equation}
   h_\theta
   =
   -4+\theta\left(-\frac{d^2}{ds^2}+s^2\right)+O(\theta^2),
\label{foundation-harper-formal}
\end{equation}
so the coefficient of $\theta$ is the harmonic oscillator, whose eigenvalues
are $2n+1$.

Formula~\eqref{foundation-harper-formal} still requires justification: a
bounded operator is being compared with an unbounded harmonic oscillator.
Helffer and Sj\"ostrand~\cite{HelfferSjostrandI,HelfferSjostrandII}
gave such a justification by semiclassical localization.  Very roughly, if the common domain $L^2(\mathbb R)$ of
$h_\theta$ and the harmonic oscillator were exhausted by a sequence of common
invariant finite-dimensional spaces $V_k$, then one could regard
\eqref{foundation-harper-formal} as the limit of the Taylor expansions of
$h_\theta|_{V_k}$.  No such sequence exists.  Instead, Helffer and
Sj\"ostrand use the spaces spanned by low-energy harmonic-oscillator
eigenfunctions.  These spaces are not invariant under $h_\theta$, but the
resulting errors are $O(h^\infty)$, with $h=\theta/(2\pi)$, and therefore do
not affect the power-series expansion.

The comparison with the Helffer--Sj\"ostrand argument should be understood
in the following sense.  Their semiclassical analysis justifies the
Wilkinson expansion after the Harper operator has been regarded as a single
semiclassical operator whose localized model is the harmonic oscillator;
the differential and multiplication variables are therefore treated
together in that model.

The representation-theoretic finite-dimensionalization used here is
organized differently.  The exact Heisenberg restriction formula identifies
the rational clock and shift matrices as genuine Bloch fibers before any
limiting argument is applied.  Thus the two basic unitaries corresponding to
translation and modulation can be studied separately before they are
combined into the Harper operator.  At each rational flux $p/q$ the resulting
Rayleigh--Schr\"odinger coefficient calculation is an ordinary finite-matrix
perturbation calculation.  This is the coefficient-calculation procedure
familiar from quantum mechanics textbooks such as Sakurai--Napolitano
\cite{Sakurai} and used in the Harper analysis of Rammal--Bellissard
\cite{RammalBellissard}.

In this limited sense, the exact rational fibers provide a mathematical and
representation-theoretic setting for the Rayleigh--Schr\"odinger coefficient
calculation in the Harper problem.  This statement does not replace the
Helffer--Sj\"ostrand localization estimates needed to identify those
coefficients with the actual band-edge asymptotics uniformly in the
large-denominator limit.  Rather, the two viewpoints justify different parts
of the argument: the present construction gives the exact generator-wise
finite-dimensional framework at rational flux, while the semiclassical
analysis controls the final uniform spectral asymptotics.

\appendix

\section*{Appendices}
\addcontentsline{toc}{section}{Appendices}
The appendices collect standard background material used in the main text.  Their purpose is expository: they recall induced representations and basic orbit-method facts for nilpotent Lie groups, mainly as a reference for the notation and examples used above.  No originality is claimed for this general representation-theoretic background, which may be found in more detail in \cite{CorwinGreenleaf}.  The brief finite-quotient discussion at the end is included only to record the normalization of the finitely additive trace functional already used in Step N-7; it is not another finite-dimensional decomposition theorem.

\section{Induced representations}\label{Inducedrep}

We briefly recall the definition of induced representations.
Although this notion is fundamental in representation theory,
it is often regarded as conceptually subtle
when first encountered.
Induced representations play a central role
in the orbit method in the sense of Kirillov
\cite{Kirillov1962,Kirillov2004}.
A personal account of Kirillov's perspective on this subject
can be found in his contribution
to the memorial volume for George Mackey \cite{Kirillov3}.

At an informal level,
an induced representation of a group $G$
from a subgroup $H$
may be viewed as a representation
constructed from a given representation of $H$,
which extends it to the ambient group
in a maximal and systematic manner.
This heuristic description,
while not a definition,
often provides a useful first intuition.

We now recall the precise definition.
Let $G$ be a unimodular locally compact group
and $H \subset G$ a closed subgroup.
Given a unitary representation $(\pi, W)$ of $H$,
the induced representation
$\rho = \mbox{Ind}_H^G \pi$
is defined on a suitable space
of $W$--valued functions on $G$,
as described below.

When $H$ is the trivial subgroup $\{e\}$,
the representation of $G$ induced from the trivial representation
reduces to the (right) regular representation $R$ of $G$,
defined by
\[
(R(g)f)(x) := f(xg),
\qquad g,x \in G,\ f \in L^2(G).
\]

More generally,
the representation space $V$ of $\rho = \mbox{Ind}_H^G \pi$
consists of measurable functions
$\phi \colon G \to W$
satisfying the equivariance condition
\[
\phi(h^{-1}g) = \pi(h)\phi(g),
\qquad h \in H,\ g \in G,
\]
and such that
\[
\int_{H\backslash G} \|\phi(g)\|_W^2 \, d\mu(g) < \infty,
\]
where $d\mu$ denotes a $G$--invariant measure on the homogeneous space
$H\backslash G$.
The action of $\rho$ is given by
\[
(\rho(g)\phi)(x) = \phi(xg),
\qquad g,x \in G.
\]

From a geometric point of view,
let $\Pi \colon G \to H\backslash G$
be the principal $H$--bundle.
The space $V$ can be identified
with the space $\Gamma(E)$ of sections
of the associated vector bundle
\begin{equation}
E := W {}_\pi\times G
\;=\;
(W \times G)/\!\sim,
\qquad
(w,x) \sim (\pi(h)w, h^{-1}x),
\quad h \in H,
\label{inducedrepresentationbundle}
\end{equation}
and the representation $\rho(g)$ acts on sections by
\[
(\rho(g)\psi)(\Pi(x)) = \psi(\Pi(x)g),
\qquad g \in G.
\]

This construction is analogous to the flat vector bundle
$E_R$ over $M$
associated with the regular representation $R$,
which has already appeared in our discussion of periodic structures.

\section{Unitary representation of nilpotent Lie groups}

\subsection{Construction of irreducible unitary representations}\label{B.1}
This section is also provided for the convenience of readers unfamiliar with the terminology of representation theories and can be viewed as an ideal model for the orbit method.

Let $G$ be a nilpotent Lie group. It is well known that irreducible unitary representations are described by the orbit method,  \`a la Kirillov \cite{Kirillov1962}, \cite{Kirillov2004}, \cite{CorwinGreenleaf}.  
All the materials in this section are quoted from \cite{CorwinGreenleaf} essentially with some minor modifications.

Let $\mathfrak{g}^{\ast}$ be the vector space dual of the Lie algebra $\mathfrak{g}$ of $G$. $G$ acts on $\mathfrak{g}^\ast$ by the coadjoint action $\mbox{Ad}^\ast(G)$. For $l \in \mathfrak{g}^\ast$ and $x \in G$, we sometimes write $l\cdot x$ for $\mbox{Ad}^\ast(x^{-1})l$; note that $l\cdot (xy)=(l\cdot x)\cdot y$.
Given $l \in \mathfrak{g}^\ast$, let $B_l$ be the bilinear form $B_l(X,Y)= l([X, Y])$ and let $\mathfrak{r}_l$ be its radical defined by 
\[\mathfrak{r}_l = \{X \in \mathfrak{g} | B_l(X,Y) = 0,\; \mbox{for all} \; Y \in \mathfrak{g}\}.
\]

Choose a maximal subordinate subalgebra (a.k.a. polarized algebra or polarization) $\mathfrak{m}$ for $l$, which is maximal in the set of subalgebras $\mathfrak{h}$ satisfying $B_l(X,Y)= 0$ for all $X,Y \in \mathfrak{h}$,  and let $M =\exp \mathfrak{m}$. Then the map of $M \to S^1 \cong U(1)$ defined by 
\[ \chi_{l,M}(\exp Y) = e^{2\pi\sqrt{-1}l(Y)}, \quad Y \in \mathfrak{m},\]
is a one-dimensional representation (i.e. unitary character) of $M$, since $B_l(\mathfrak{m}, \mathfrak{m})= 0$. 
We may therefore form the induced representation $\pi_{l,M}  = \mbox{Ind}_M^G\chi_{l,M}$.
The following results describe the unitary dual $\widehat{G}$, the set of (unitary) equivalence classes of irreducible unitary representations of $G$, in terms of these induced representations. 

\begin{theorem}\label{nilpLieirredrep} Let $l \in \mathfrak{g}^\ast$. Then we have:
\begin{description}
\item[{\rm (1)}]{\rm (2.2.1 Theorem in \cite{CorwinGreenleaf})} There exists a maximal subordinate subalgebra $\mathfrak{m}$ for $l$ such that $\pi_{l,M}$ is irreducible. 
\item[{\rm (2)}]{\rm (2.2.2 Theorem in \cite{CorwinGreenleaf})} Let $\mathfrak{m}, \mathfrak{m}'$  be two maximal subordinate subalgebras for $l$. Then $\pi_{l,M} \cong \pi_{l,M'}$. $($unitarily equivalent$)$. Thus, $\pi_{l,M}$ will be denoted by $\pi_l$ for simplicity when only the equivalence class is relevant.
\item[{\rm (3)}]{\rm (2.2.3 Theorem in \cite{CorwinGreenleaf})} Let $\pi$ be any irreducible unitary representation of $G$. Then there is an $l \in \mathfrak{g}^\ast$ such that $\pi_l \cong \pi$.
\item[{\rm (4)}]{\rm (2.2.4 Theorem in \cite{CorwinGreenleaf})} Let $l, l' \in \mathfrak{g}^\ast$ . Then $\pi_l \cong \pi_{l'}$ (unitarily equivalent) if and only if $l$ and $l'$ are in the same $\mbox{Ad}^\ast(G)$-orbit in $\mathfrak{g}^\ast$.
\end{description}
\end{theorem}

We may summarize these results as follows: the map $l \mapsto \pi_{l,M}$ is independent of $\mathfrak{m}$ and gives a bijection between the orbit space $\mathfrak{g}^\ast /\mbox{Ad}^\ast(G)$ and the unitary dual $\widehat{G}$.

\subsection{Coadjoint orbits}

Concerning the Fourier inversion formula, Theorem \ref{FourierinversionLie}, we need only the information of the generic orbits. The following are the first instances describing them. 

\begin{theorem}[{\rm 3.1.6 Theorem in \cite{CorwinGreenleaf}}]\label{corwin316} Let $G$ act unipotently on $V$ and let $\{e_1,\ldots,e_m\}$ be a Jordan-H\"older basis for this action, namely, $V_j := \mathbb{R}$\mbox{-span}$\{e_{j+1},\ldots,e_m\}$ is $G$-stable for all $j$. Then there are disjoint sets of indices $S$, $T$ with $S \cup T=\{1,2,\ldots,m\}$, a Zariski open set $U \subset V$, and rational functions $Q_1(x, t),\ldots, Q_m(x, t)$ of the variables $(x, t) = (x_1 ,\ldots, x_m, t_1 ,\ldots, t_k)$ where $k = \mbox{card}(S)$, with the following properties: if $S = \{j_1 < \cdots <j_k\}$ and if we identify $x \in \mathbb{R}^m$ with $v = \sum_{i= 1}^m x_ie_i$, then 

\begin{description}
\item[{\rm (i)}]	The functions $Q_i(x, t)$ are rational nonsingular on $U \times \mathbb{R}^k$. For fixed $x$, they are polynomials in $t$.
\item[{\rm (ii)}] For each $v = \sum_{i= 1}^m x_ie_i$ in $U$, $Q(x, t) = \sum_{i= 1}^m Q_i(x, t)e_i$ maps $\mathbb{R}^k$ diffeomorphically onto the orbit $G \cdot v$.
\item[{\rm (iii)}] For fixed $x$, the function $Q_j(x, t)$ depends only on those $t_i$ such that $j_i \leq j$.
\item[{\rm (iv)}] If $j \not\in S$ then $Q_j(x, t) =x_j+ R(x_1,\ldots, x_{j-1}, t_1,\ldots, t_i)$ where $i$ is the largest index such that $j_i <j$ and $R$ is rational. Moreover, $Q_1(x, t) = x_1$.
\item[{\rm (v)}] $Q_{j_i}(x,t)= t_i+x_{j_i}+R(x_1 ,\ldots,,x_{j_i-1}, t_1,\ldots,t_{i-1} )$ where $R$ is rational.
\end{description}
\end{theorem}

Next, by changing the variables $(x, t) = (x_1,\ldots, x_m, t_1,\ldots, t_k)$ to $(x, u) = (x_1,\ldots, x_m, u_1,\ldots,u_k)$, we can simplify the item (v) in this theorem to the following item (v) in following theorem.

\begin{theorem}[{\rm 3.1.8 Corollary in \cite{CorwinGreenleaf}}]\label{corwin318}
Let $G$ act unipotently on $V$, let $\{e_1,\ldots,e_m\}$ be a Jordan--H\"older basis, and let $d_j$ be the generic dimension of the $G$-orbits in $V/V_j$, where $V_j=\mathbb{R}\mbox{-span}\{e_{j+1},\ldots,e_m\}$.  Let $k=d_m$ be the generic dimension of the orbits in $V$, let $U$ be the set of $v\in V$ for which the orbit dimension in every quotient $V/V_j$ is $d_j$, let $S=\{j_1<\cdots<j_k\}$ be the set of indices for which $d_j\neq d_{j-1}$, and put $T=\{1,2,\ldots,m\}\setminus S$.  If we identify $v=\sum_{i=1}^m x_i e_i$ with $x\in\mathbb{R}^m$, there are rational functions $P_1,\ldots,P_m$ in the $m+k$ variables $(x,u)=(x_1,\ldots,x_m,u_1,\ldots,u_k)$ such that the following hold.
\begin{description}
\item[{\rm (i)}] The $P_1,\ldots,P_m$ are rational and nonsingular on $U\times\mathbb{R}^k$.  For fixed $x$, they are polynomials in $u$.
\item[{\rm (ii)}] If $v=\sum_{i=1}^m x_i e_i$, the map $P(x,u)=\sum_{j=1}^mP_j(x,u)e_j$ is a diffeomorphism of $\mathbb{R}^k$ onto the orbit $G\cdot v$.
\item[{\rm (iii)}] For fixed $x$, the function $P_j(x,u)$ depends only on those $u_i$ for which $j_i\leq j$.
\item[{\rm (iv)}] If $j\notin S$, then $P_j(x,u)=x_j+R(x_1,\ldots,x_{j-1},u_1,\ldots,u_i)$, where $i$ is the largest index such that $j_i<j$ and $R$ is rational.  Moreover, $P_1(x,u)=x_1$.
\item[{\rm (v)}] $P_{j_i}(x,u)=u_i$, $1\leq i\leq k$.
\end{description}
Finally, $U$ is $G$-invariant and, for fixed $u\in\mathbb{R}^k$, each $P_j(x,u)$ is a rational nonsingular function on $U$ which is constant on $G$-orbits.
\end{theorem}

The above two theorems are also expressed in different forms: 

\begin{theorem}[{\rm 3.1.9 Theorem in \cite{CorwinGreenleaf}}]\label{corwin319}
Given a Jordan-H\"older basis $\{e_1,\ldots, e_m \}$ for a unipotent action of $G$ on vector space $V$, define the generic dimensions $d_i$ $(1\leq i\leq m)$, a partition $S\cup T = \{1, 2,\ldots, m \}$, and the $G$-invariant set $U$ of generic orbits as in Theorems \ref{corwin316} and \ref{corwin318}.
Let $V_S =\mathbb{R}\mbox{-span}\{e_i: i \in S\}$, $V_T = \mathbb{R}\mbox{-span}\{e_i: i \in T\}$ and let $p_S, p_T$ be the projections of $V$ to $V_S, V_T$. Then
\begin{description}
\item[{\rm (i)}] Every $G$-orbit in $U$ meets $V_T$ in a unique point. In particular, $U\cap V_T$ is nonempty and Zariski open in $V_T$.
There is a map $\psi: (U\cap V_T) \times V_S \to U$ such that

\item[{\rm (ii)}] $\psi$ is a rational, nonsingular bijection with rational, nonsingular inverse,
\item[{\rm (iii)}] For each $v \in V_T\cap U$, the map $P_v(\cdot)= p_T(\psi(v,\cdot ))$ from $V_S$ into $V_T$ is a polynomial, and the orbit $G \cdot v$ is its graph,
\item[{\rm (iv)}] The Jacobian determinant of $\psi$ is identically $1$.
\end{description}
\end{theorem}

\subsection{Decomposition of the right regular representation}\label{B.3}
We recall the Fourier inversion formula and the Kirillov character formula.
The former describes the content of a decomposition of the right regular representation into a direct integral of irreducible unitary representations of the nilpotent Lie group $G$. To define this formula, we need to compute the trace of the Fourier transform $\pi_l(\phi)$ defined by 
\[  \pi_l(\phi) := \int_G\pi_l(\sigma )\phi(\sigma)d\sigma \quad \phi \in \mathcal{S}(G)\]
where $l \in \mathfrak{g}^\ast$ and $\mathcal{S}(G)$ is the space of rapidly decreasing functions on $G$ (a.k.a. Schwartz class).

The Kirillov character formula is a tool for the computation of the trace $\mbox{Tr}(\pi_l(\phi))$ of $\pi_l(\phi)$. We first recall several results to compute them.  

First, we recall notions of weak or strong Malcev basis of a nilpotent Lie group $G$ and its Lie algebra $\mathfrak{g}$.

\begin{definitiontheorem}[1.1.13 Theorem and Note in \cite{CorwinGreenleaf}]\label{malcevbasis}
Let $\mathfrak{g}$ be a nilpotent Lie algebra, and let $\mathfrak{g}_1 \supset \mathfrak{g}_2 \supset \cdots \supset \mathfrak{g}_k$ be subalgebras, with ${\rm dim}\; \mathfrak{g}_i = m_i$ and ${\rm dim}\; \mathfrak{g} = n$.
\begin{description}
\item[{\rm (a)}] $\mathfrak{g}$ has a basis $\{X_1, \ldots, X_n\}$ such that
\begin{description}
\item[{\rm (i)}] for each $m$, $\mathfrak{h}_m = \mathbb{R}\mbox{-span}\{ X_1, \ldots, X_m\}$ is a subalgebra of $\mathfrak{g}$,
\item[{\rm (ii)}] for $1 \leq j \leq k$, $\mathfrak{h}_{m_j} = \mathfrak{g}_j$.
\end{description}
This basis is called a weak Malcev basis for $\mathfrak{g}$ (through $\mathfrak{g}_1,\ldots, \mathfrak{g}_k$).

\item[{\rm (b)}] If the $\mathfrak{g}_i$ are ideals of $\mathfrak{g}$, then one can pick the $X_i$ so that (i) is replaced by
\begin{description}
\item[{\rm (iii)}] for each $m$, $\mathfrak{h}_m = \mathbb{R}\mbox{-span}\{ X_1, \ldots, X_m\}$ is an ideal of $\mathfrak{g}$.
\end{description}
This basis is called a strong Malcev basis for $\mathfrak{g}$ (through $\mathfrak{g}_1,\ldots, \mathfrak{g}_k$).
\end{description}
\end{definitiontheorem}

\begin{theorem}[4.2.1 Theorem in \cite{CorwinGreenleaf}]\label{nilpotenttrace} Let $\pi =\pi_l$ be an irreducible representation of a nilpotent Lie group $G$, let $\mathfrak{m}$ be a polarization for $l$, and model $\pi$ in $L^2(\mathbb{R}^k)$ {\rm (}i.e. representation space of $\pi${\rm )} using any weak
Malcev basis through $\mathfrak{m}$. If $\phi \in \mathcal{S}(G)$, then $\pi_l(\phi)$ is trace class and
\begin{equation*}
\pi_l(\phi) f(s) = \int_{\mathbb{R}^k}K_\phi (s, t)f(t) dt, \quad \mbox{for all} \quad f \in L^2(\mathbb{R}^k )
\end{equation*}
where $K_\phi \in \mathcal{S}(\mathbb{R}^k \times \mathbb{R}^k)$. Furthermore, $\mbox{Tr}\;\pi_l(\phi)$ is given by
\begin{equation*}
\mbox{Tr}\;\pi_l(\phi) = \int_{\mathbb{R}^k}K_\phi (s, s)ds. \quad \mbox{(absolutely convergent)} 
\end{equation*}
\end{theorem}

A detailed expression of the kernel function $K_\phi$ is given as follows:
Let $\{X_1,\ldots,X_n\}$ be the weak Malcev basis taken in Theorem \ref{nilpotenttrace} and put $p = n-k = \mbox{dim}\;\mathfrak{m}$. Define polynomial maps $\gamma: \mathbb{R}^n \to G$, $\alpha:\mathbb{R}^p\to M =\exp(\mathfrak{m})$, $\beta:\mathbb{R}^k \to G$ by
\begin{align*}
&{} \gamma(s,t) = \exp s_1X_1\cdot\cdots\cdot\exp s_pX_p\cdot\exp t_1X_{p+1}\cdot\cdots\cdot\exp t_kX_n \\  
&{} \alpha(s) = \gamma(s,0),\quad \beta(t) = \gamma(0,t),
\end{align*}
and let $dg$, $dm$, $d\dot{g}$ be the invariant measures on $G$, $M$, $M\backslash G$ determined by Lebesgue measures $dsdt$, $ds$, $dt$ through maps. $\gamma$, $\alpha$, $\beta$ in the above. 

Then, we have the following:
\begin{proposition}[4.2.2 Proposition in \cite{CorwinGreenleaf}]\label{nilpLietrace}
If we take the standard basis realization of $\pi =\pi_l$ in $L^2(\mathbb{R}^k 
)$ 
relative to the given Malcev basis, the kernel $K_\phi$ has the form 
\begin{equation}
K_\phi(t',t) = \int_M\chi_l(m)\phi(\beta(t')^{-1}m\beta(t))dm \quad \mbox{{\rm (} absolutely convergent {\rm )}}  
\end{equation}
where $\chi_l(\exp Y) = e^{2\pi\sqrt{-1}l(Y)}$ and $\beta$ is the map defined above. 
\end{proposition}

\begin{remark} If we directly relate representations $G$ to $\Gamma$, then we need to identify a function $\phi:\Gamma \to \mathbb{C}$ with that $\widetilde{\phi}:G \to \mathbb{C}$. 
Although there seem to be several methods to perform it, if we take an easy way to define $\widetilde{\phi}(g) = \phi(\sigma)$ if $x \in \sigma\mathcal{D}$ with a fundamental domain $\mathcal{D}$ of the canonical projection $G \to G/\Gamma$, then $\widetilde{\phi} \in L^1(G)$ but $\not\in \mathcal{S}(G)$ and thus, $\pi(\widetilde{\phi})$ is a compact operator but not in trace class in general. (see Appendix~\ref{AppendixD}).
\end{remark}

Once we proceed with the computation further, we arrive at the following formula.
Given any Euclidean measure $dX$ on $\mathfrak{g}$, define the Euclidean Fourier transforms $\hat{f}$ of functions $f$ on $G$ to be
\[ \hat{f}(l) = \int_\mathfrak{g}e^{2\pi\sqrt{-1}l(X)}f(\exp X)dX \quad \mbox{for all}\quad l \in \mathfrak{g}^\ast.
\]

Each coadjoint orbit $\mathcal{O}_l=\mbox{Ad}^\ast G(l)$ of $l$, being a closed set, 
carries an invariant measure $\mu$ that is unique up to a scalar multiple because $\mathcal{O}_l \simeq R_l\backslash G$, 
where $R_l = \mbox{Stab}_G(l) = \{x \in G|(\mbox{Ad}^\ast x)l = l\}$. 

\begin{theorem}[4.2.4 Theorem in \cite{CorwinGreenleaf}, Kirillov character formula]\label{nilpLiecharcter} If $\pi$ is an irreducible unitary representation of a nilpotent Lie group $G$, corresponding to the coadjoint orbit $\mathcal{O}_l=\mbox{Ad}^\ast G(l) \subset \mathfrak{g}^\ast$, there is a unique 
choice of the invariant measure $\mu$ on $\mathcal{O}_l$ such that
\[ \mbox{Tr}\;\pi(\phi) = \int_{\mathcal{O}_l}\hat{\phi}(\xi)\,\mu(d\xi) \quad \mbox{for all}\quad \phi \in \mathcal{S}(G),
\]
The integral is absolutely convergent. 
\end{theorem}

Under the above preparation, we can state the following formula. 

\begin{theorem}[4.3.9 Theorem in \cite{CorwinGreenleaf}, Fourier Inversion formula]\label{FourierinversionLie}
Let $\{X_1,\ldots,X_n\}$ be a strong Malcev basis for a nilpotent Lie algebra $\mathfrak{g}$ and 
let $\{l_1,\ldots, l_n\}$ be the dual basis for $\mathfrak{g}^\ast$. 
Take a set $U$ of generic coadjoint orbits, index sets 
\[ S = \{i_1 < \cdots < i_{2k}\} \] 
and $T$ which will appear in Theorem \ref{corwin319}, and (absolute value of) Pfaffian $|\mbox{Pf}(l)|$ satisfies
\[ |\mbox{Pf}(l)|^2 = \det \;B, \quad B= (B_{jl}) = (B_l(X_{i_j},X_{i_l})) := (l([X_{i_j},X_{i_l}])) \quad j,l =1,\ldots, 2k \] 
Then for $\phi \in \mathcal{S}( G)$, $\phi(e)$ is given by an absolutely convergent integral 
\begin{equation}
\phi(e) = \int_{U\cap V_T}|\mbox{Pf}(l)|\mbox{Tr}\;\pi_l(\phi) dl \label{B2}
\end{equation}
where $dl$ is the Euclidean measure on $V_T = \mathbb{R}\mbox{-span}\{l_i| i \in T\}$ such that the cube determined by $\{l_i| i \in T\}$ has mass $1$. 
\end{theorem}

\begin{remark}
In the above theorem,
the existence of a strong Malcev basis
is assumed for simplicity.
When such a basis is not available,
there are two standard ways to proceed.

First, one may work with a suitably chosen weak Malcev basis
and adjust the polarization at each step,
as in the inductive procedure described
in Subsection~\ref{choicepolarization}.
This allows one to carry through the construction
at the cost of additional technical complexity.

Alternatively, one may fix a more rigid structure
from the outset by passing to a stratified model
of the Lie algebra
and choosing a canonical polarization,
for instance via Vergne's construction
based on a chain of ideals~\cite{Vergne}.
This approach avoids the step--by--step modification
of polarizations and provides a uniform framework
for subsequent analytic constructions,
such as the definition of canonical hypo--elliptic operators
(see Subsection~\ref{constracthypoelliptic}).

Both approaches ultimately lead to the same representation--theoretic objects,
but the latter has the advantage of being globally well adapted
to the analytic structures developed in the paper.
\end{remark}

\subsection{Three examples} 


For nilpotent Lie group $G$, the exponential map $\exp: \mathfrak{g} \to G$ is a diffeomorphism, therefore we can define $X \ast Y \in \mathfrak{g}$ for all $X,Y \in \mathfrak{g}$ by the relation
\[\exp(X\ast Y) = \exp X\cdot\exp Y
\]
If $\mathfrak{g}$ is equipped with coordinates associated with a linear basis, the corresponding coordinates in $G$ will be called exponential coordinates, or canonical coordinates of the first kind.

\subsubsection[(2n+1)-dimensional real Heisenberg group Hn]{$(2n+1)$-dimensional real Heisenberg group $H_n$}\label{2n+1dimensionalHeisenberg}
We summarize Examples 1.1.2, 1.2.4, 1.3.9, 2.2.6, 4.3.11 in \cite{CorwinGreenleaf}. Note that $H_n$ in \cite{CorwinGreenleaf} is the same as the Heisenberg-Lie group $\mbox{Heis}_n(\mathbb{R})$ in the notation in the body of this paper.
 
\paragraph{Definition}
We define $\mathfrak{h}_n$, the $(2n+ 1)$-dimensional Heisenberg algebra,
to be the Lie algebra with basis $\{Z, Y_1,\ldots, Y_n, X_1,\ldots, X_n\}$, whose pairwise Lie brackets are equal to zero except for
$[X_i,Y_j]=\delta_{ij}Z$, $1\leq i,j\leq n$. It is a two-step nilpotent Lie algebra. One way to realize it as matrix algebra is to let $zZ + \sum_{i=1}^n (x_iX_i +y_iY_i)$ correspond to the $(n + 2) \times (n + 2)$ matrix
\[ \left(\begin{array}{ccccc} 0 & x_1 & \ldots & x_n & z \\ 
0 &  0   &  \ldots   &  0   & y_1 \\				
{}  &  {}   &  {}     &  {}   & \vdots \\
0 & {}	&   {}     &  0  & y_n \\
0 & 0	&   \ldots  &  0  & 0
\end{array}\right)
\]				
Note that $\{Z, Y, X\} = \{Z, Y_1,\ldots, Y_n, X_1,\ldots, X_n\}$ is a strong Malcev basis for $\mathfrak{h}_n$ through $\mathfrak{g}_1, \mathfrak{g}_2, \mathfrak{g}_3$ 
where $\mathfrak{g}_1$ is the center of $\mathfrak{h}_n$, which is generated by $Z$, $\mathfrak{g}_2$ is generated by $Z,Y_1,\ldots, Y_n$ and $\mathfrak{g}_3 = \mathfrak{h}_n$. These are ideals of $\mathfrak{h}_n$.

We denote a typical element of $\mathfrak{h}_n$ by $zZ + \sum_{i=1}^n (x_iX_i+y_iY_i) = (z,y,x)$, with $z\in \mathbb{R}$ and $x, y \in \mathbb{R}^n$. Using exponential coordinates for $H_n$, we get
\[
(z,y,x)\ast(z',y',x')
=
\left(z+z'+\frac12(x\cdot y'-y\cdot x'),\ y+y',\ x+x'\right).
\]
where $x\cdot y$ is the usual inner product on $\mathbb{R}^n$. Similarly,
\[
({\rm Ad}(\exp(z, y, x)))(z', y', x')=(z' + x\cdot y' -y\cdot x',y',x').
\]
If we can use the matrix representation of $\mathfrak{h}_n$, then
\begin{equation}
\exp(z, y, x) = \left(\begin{array}{ccccc}
1 & x_1 & \ldots & x_n & z^\ast \\ 
0  & 1  & \ldots &  0   & y_1 \\				
  &     &.       &     & \vdots \\
0 &	&        &  1   & y_n \\
0 & 0	&  \ldots &  0  & 1
\end{array}\right) =: [z^\ast,y,x] \quad {\rm where}\quad
 z^\ast = z + \frac12 x\cdot y. \label{expnentiltomatrix}
\end{equation}

\paragraph{Coadjoint orbits} 
Next, we consider coadjoint orbits in the dual space $\mathfrak{h}_n^{\ast}$ and associated irreducible unitary representations of $H_n$. We denote by $\{Z^\ast, Y_1^\ast,\ldots, Y_n^\ast, X_1^\ast,\ldots, X_n^\ast\}$ the dual basis of $\{Z, Y_1,\ldots, Y_n, X_1,\ldots, X_n\}$. An element $l \in \mathfrak{g}^\ast$ is written as 
\[ l = l_{\alpha,\beta,\gamma} := \gamma Z^\ast + \sum_{i=1}^n (\beta_i Y_i^\ast +\alpha_i X_i^\ast)
\]
For $w = (z,y,x) \in H_n$ and $W = cZ + \sum_{i=1}^n b_iY_i + a_iX_i \in \mathfrak{h}_n$, we can compute as 
\[ {\rm Ad}^\ast(w)(l_{\alpha,\beta,\gamma})(W) = {\rm Ad}^\ast(w)(l)(W) =  l({\rm Ad}(w^{-1})W) = l_{\alpha+\gamma y,\beta-\gamma x, \gamma}  
\]
Thus, coadjoint orbits are classified into the following two categories:
\begin{description}
\item[{\rm (i)}]($2n$)-dimensional orbits with $\gamma \neq 0$)
\[ \gamma Z^\ast + \mathfrak{z}^{\perp} =\{l_{\alpha+\gamma y,\beta-\gamma x, \gamma}| x,y \in \mathbb{R}^n\}
\]
where $\mathfrak{z}^{\perp} = \{l \in \mathfrak{g}^\ast | l(Z) = 0\} = \mathbb{R}Y^\ast + \mathbb{R}X^\ast$. 

\item[{\rm (ii)}]($0$-dimensional orbits with $\gamma = 0$) Each $0$-dimensional orbit is a point in $\mathfrak{z}^{\perp}$.
\end{description}

\paragraph{Irreducible representations associated with coadjoint orbits}
For the point orbits, $\mathfrak{g} = \mathfrak{h}_n$ is the radical $\mathfrak{r}_l$ and the only polarizing subalgebra $\mathfrak{m}$. If $l(Z) = \gamma \neq 0$, then $\mathfrak{r}_l = \mathbb{R}Z$ and there are many polarizing subalgebras $\mathfrak{m}$; for example 
\[ \mathfrak{m} = \mathbb{R}Z\oplus\mathbb{R}\mbox{-span}\{Y_1,\ldots, Y_n\} \quad {\rm or}\quad \mathfrak{m} = \mathbb{R}Z\oplus\mathbb{R}\mbox{-span}\{X_1,\ldots, X_n\} 
\]
Since the unitary equivalence class of $\pi_{l,M}$ is independent of the choice of orbit representative $l$ and polarizing subalgebra $\mathfrak{m}$, it is convenient to take an element $l \in \gamma Z^\ast+ \mathfrak{z}^{\perp}$ as $l=\lambda Z^\ast (\lambda = \gamma \neq 0)$ and $\mathfrak{m}$ as above. 
Then
\[ \chi_l\left(\exp\left(zZ+\sum_{j=1}^n y_jY_j\right)\right) =e^{2\pi\sqrt{-1}\lambda z}
\]
defines a character on the normal subgroup $M = \exp \mathfrak{m}$ of $G$, which induces to a representation $\pi_l = \mbox{Ind}_M^G(\chi_l)$ on $G$. A description of $\pi_l$ as an action on $L^2(M\backslash G, \mathbb{C})$ is more enlightening than the standard model as an action on functions on $G$ varying like $\chi_l$ along $M$-cosets. The calculation of this new model is typical of induced representations and will often arise. Let
\[
A= \Big\{\exp\left(\sum_{j=1}^n t_jX_j\right)\Big| t_j \in \mathbb{R}\Big\}.
\]
Then $A$ is a transversal for $M\backslash G$ since $M$ is normal, and Lebesgue measure $dt = dt_1\ldots dt_n$ transfers to  a  right-invariant measure 
$d\dot{g}$ on $M\backslash G$. Using $dt$ and $d\dot{g}$ to define $L^2$-norms, and identifying $\mathbb{R}^n$ with $A$ in the obvious way, 
the restriction map $f \to f|A$ becomes  an  isometry  from  the representation space $\mathcal{H}_{\pi_l}$ of $\pi_l$ to $L^2(\mathbb{R}^n)$. To compute the action of $\pi_l$ in this $L^2(\mathbb{R}^n)$ model,
the fundamental problem in computing this action is to split products $(0, 0, t)(z, y, x)$ in the form $h\cdot (0, 0, t')$ where $h \in M$. From the multiplication law
\[ (z,y,x)\cdot (z',y',x') = \left(z+z'-\frac{y\cdot x'}{2}+\frac{x\cdot y'}{2}, y+y',x+x'\right),
\]
we see that
\begin{equation}
(0,0,t)\cdot (z,y,x) = \left(z+\frac{t\cdot y}{2}, y,x+t\right) =  \left(z+ t\cdot y+\frac{x\cdot y}{2}, y,0\right)(0,0,t+x).  \label{exponetialmultiplication}
\end{equation}
Hence if $f\in\mathcal{H}_{\pi_l}$, then
\begin{align*}
\pi_l(z,y,x)f(0,0,t)
&=f\left(\left(z+t\cdot y+\frac{x\cdot y}{2},y,0\right)
          \cdot(0,0,t+x)\right)\\
&=e^{2\pi\sqrt{-1}\lambda(z+t\cdot y+(1/2)x\cdot y)}
  f(0,0,t+x).
\end{align*}
Thus the action on $\tilde{f} \in L^2(\mathbb{R}^n)$ is given by
\begin{equation}
\pi_l(z, y, x)\tilde{f}(t)= e^{2\pi\sqrt{-1}\lambda (z+t\cdot y+(1/2)x\cdot y)}\tilde{f}(t+x) \quad (\lambda \neq 0).\label{2n+1heisenbergirred}
\end{equation}
Note again that this is different from the formula (\ref{schrep}) in the case when $n=1$ for the reason explained there. In fact, in the matrix expression (\ref{expnentiltomatrix}), equality (\ref{exponetialmultiplication}) is changed to 
\begin{equation}
[0,0,t]\cdot [z^\ast,y,x] = \left[z^\ast+t\cdot y, y,x+t\right] =  \left[z^\ast+ t\cdot y, y,0\right][0,0,t+x]. \label{matrixmultiplication}
\end{equation}
Kirillov's theory says that every irreducible representation of $H_n$ is obtained in this way. 

\paragraph{Fourier inversion formula}
Finally, we recall the Fourier inversion formula. 
Recall that Malcev basis $\{Z, Y_1,\ldots, Y_n, X_1,\ldots,X_n\}$ satisfies
\[  [X_i,Y_j] = \delta_{ij}Z,\quad 1 \leq i,j \leq n.
\] 
If $l_1,\ldots,l_{2n+1}$ is the basis of $\mathfrak{g}^\ast$, the index set partitions as
\[
S=\{2,3,\ldots,2n+1\},\qquad T=\{1\}.
\]
Hence $V_T=\mathbb{R}l_1$ is a cross-section for the generic orbits
$U=\{l:l(Z)\neq0\}$, and the Pfaffian satisfies
\[
|\operatorname{Pf}(l)|=|l(Z)|^n.
\]
Identifying $V_T=\{\alpha l_1:\alpha\in\mathbb{R}\}$ and using the normalized Lebesgue measure $d\alpha$, Fourier inversion becomes
\begin{equation}
\phi(\sigma)
=
\int_{\mathbb{R}}
\operatorname{Tr}\!\left(
\pi_{\alpha l_1}(\sigma^{-1})\pi_{\alpha l_1}(\phi)
\right)
|\alpha|^n\,d\alpha,
\qquad \phi\in\mathcal{S}(H_n).
\label{2n+1heisenbergplancherel}
\end{equation}

\paragraph{Associated hypo--elliptic operator}

Let $\Gamma=\mathrm{Heis}_{2n+1}(\mathbb Z)$ and
$G=\mathrm{Heis}_{2n+1}(\mathbb R)$.  For a generic coadjoint orbit, the
associated irreducible unitary representation is realized on
$L^2(\mathbb R^n)$, and the canonical hypo--elliptic operator is
\[
\mathcal H
=
\sum_{j=1}^n
\left(
-\frac{\partial^2}{\partial t_j^2}+4\pi^2t_j^2
\right),
\]
which is the standard $n$--dimensional harmonic oscillator.

\subsubsection[Engel group E4]{Engel group $E_4$}
We summarize Examples 1.1.3, 1.2.5, 1.3.10, 2.2.7, 4.3.12  in \cite{CorwinGreenleaf}.
\paragraph{Definition}
We define the Engel Lie algebra $\mathfrak{e}_4$ to be the $4$-dimensional Lie algebra spanned by $W,X,Y,Z$ with
\[
[W,X]=Y,\qquad [W,Y]=Z,\qquad
[X,Y]=[W,Z]=[X,Z]=[Y,Z]=0.
\]
This one is a $3$ step nilpotent Lie algebra. One realization of matrix algebra is obtained by letting $wW+xX+yY+zZ$ correspond to the $4 \times 4$ matrix.

\[ 
\left(\begin{array}{cccc}
0 & w & 0 & z  \\ 
0 & 0 & w & y  \\				
0 & 0 & 0 & x  \\
0 & 0 & 0 & 0  \end{array}\right)
\]

\paragraph{Coadjoint orbits}
The basis $\{Z,Y,X,W\}$ is a strong Malcev, but we shall use the exponential coordinates on $E_4:= \exp \mathfrak{e}_4$,
\[
(z,y,x,w) = \exp(zZ + xX+ yY+ wW)
\]
below. The Campbell-Baker-Hausdorff formula yields the multiplication law in these coordinates:
\begin{align}
&{}(z,y,x,w)\ast(z',y',x',w') \notag\\
&=\Bigl(
 z+z'+\frac12(wy'-yw')
   +\frac1{12}(wx'-xw')(w-w'),\notag\\
&\hspace{31mm}
 y+y'+\frac12(wx'-xw'),\ x+x',\ w+w'
 \Bigr).
\label{engelmultiplication}
\end{align}
The adjoint action is given by
\[
\operatorname{Ad}_{(z,y,x,w)}(z',y',x',w')
=
\Bigl(
 z'+wy'-yw'+\frac12 w(wx'-xw'),
 y'+wx'-xw',\ x',\ w'
\Bigr).
\]

Relative to the given basis in $\mathfrak{e}_4$, let $\{Z^\ast=l_1,Y^\ast=l_2,X^\ast=l_3,W^\ast=l_4\}$ be the dual basis of $\mathfrak{e}_4^\ast$.  If $l= \delta l_1 + \gamma l_2 + \beta l_3 + \alpha l_4$, the coadjoint orbits and convenient representatives are
\begin{description}
\item[{\rm (i)}]  $\mathcal{O}_{\delta,\beta} = {\rm Ad}^\ast (G)(\delta l_1 +  \beta l_3) = \{ \delta l_1 + tl_2 + ( \beta + t^2/2\delta )l_3+ sl_4 | s, t \in \mathbb{R} \}$, if $\delta \neq 0$,
\item[{\rm (ii)}] $\mathcal{O}_\gamma = {\rm Ad}^\ast(G)(\gamma l_2)= \gamma l_2+ \mathbb{R}l_3+ \mathbb{R}l_4$, where $\gamma \neq 0$,
\item[{\rm (iii)}] $\mathcal{O}_{\alpha,\beta} = {\rm Ad}^\ast (G)(\beta l_3 + \alpha l_4)= \{\beta l_3+ \alpha l_4\}$, where $\alpha, \beta \in \mathbb{R}$ (one point orbit)
\end{description}

\paragraph{Irreducible representations associated with generic orbits}
The generic orbits are parabolic cylinders $\mathcal{O}_{\delta,\beta}$ with $\delta \neq 0$. For the representatives $l_{\delta,\beta} = \delta l_1 + \beta l_3$, the radical is $\mathfrak{r}_l=\mathbb{R}Z+\mathbb{R}X$; hence every polarizing subalgebra $\mathfrak{m}$ is three-dimensional. 
Since $\mathfrak{m} = \mathbb{R}Z +\mathbb{R}Y+ \mathbb{R}X$ is an abelian ideal in $\mathfrak{e}_4$ of the correct dimension, it is a polarizing subalgebra for each $l_{\delta,\beta}$.

Write $(z,y,x,w) = \exp(zZ+yY+xX+wW)$. Obviously $M = \exp(\mathfrak{m}) = \{(z,y,x,0)| z, y, x \in \mathbb{R}\}$ and the character $\chi_{\delta,\beta} := \chi_{l_{\delta,\beta}}$ on $M$ is
\[
\chi_{\delta,\beta}(z,y,x,0)= e^{2\pi\sqrt{-1}(\beta x +\delta z)}
\]
Also $M\backslash G$ with $G=E_4$ has $\Sigma = \exp{\mathbb{R}W}$ as a cross section, which we identify with $\mathbb{R}$; 
Lebesgue measure $dt$ on $\mathbb{R}$ gives a right-invariant measure $d\dot{g}$ on $M\backslash G$.
Using these measures to define $L^2$-norms we have an isometry between the representation space $\mathcal{H}_{\delta,\beta}$ of $\pi_{\delta,\beta} = {\rm Ind}_M^G(\chi_{\delta,\beta})$ and $L^2(\mathbb{R})$. 
We compute the action of $\pi_{\delta,\beta}$ modeled in $L^2(\mathbb{R})$ from the multiplication law given in (\ref{engelmultiplication}); first note that
\begin{align*}
&{}(0, 0, 0, t)\cdot(z, y, x, w) \\
&{} \quad  =(z +\frac12 ty+ \frac{1}{12}[t^2 x - twx ], y+ \frac12 tx, x, t + w) \\
&{} \quad =(z + ty + \frac12 t^2x +\frac12 twx + \frac12 wy + \frac16 w^2x, y+ tx +\frac12 wx, x, 0)\cdot(0, 0, 0, t + w).
\end{align*}
Thus if $f \in \mathcal{H}_{\delta,\beta}$
\begin{align*}
&{}[\pi_{\delta,\beta}(z,y,x,w)f](0, 0, 0, t) = f((0, 0, 0, t)\cdot (z, y, x, w)) \\
&{}= e^{2\pi\sqrt{-1}\beta x}e^{2\pi\sqrt{-1}\delta(z+ty+(1/2)t^2x +(1/2)twx +(1/2)wy + (1/6)w^2x)}f(0,0,0,t+w)
\end{align*}
and the action  modeled  in $L^2(\mathbb{R})$ is
\begin{align}
&{}[\pi_{\delta,\beta}(z,y,x,w)\tilde{f}](t) \notag \\
&{}= e^{2\pi\sqrt{-1}\beta x}e^{2\pi\sqrt{-1}\delta(z+ty+(1/2)t^2x +(1/2)twx +(1/2)wy + (1/6)w^2x)}\tilde{f}(t+ w). \label{engelirred}	
\end{align}
This action is easier to visualize if one examines the action of individual one-parameter subgroups. For instance,
\begin{align*}
&{}[\pi_{\delta,\beta}(\exp xX)\tilde{f}](t) = e^{2\pi\sqrt{-1}\beta x}e^{2\pi\sqrt{-1}\delta(1/2)t^2x}\tilde{f}(t),\\
&{}  [\pi_{\delta,\beta}(\exp wW)\tilde{f}](t) = \tilde{f}(t+w).
\end{align*}

We omit the description of the other orbits and irreducible representations since they do not appear in the Fourier inversion formula below (cf. 2.2.7 Example in \cite{CorwinGreenleaf}).

\paragraph{Fourier inversion formula}
Recall that $\{l_1,l_2,l_3,l_4\}$ is the dual basis of the strong Malcev basis $\{Z,Y,X,W\}$.  Then
\[
S=\{2,4\},\qquad T=\{1,3\},
\]
and
\begin{align*}
\mbox{Generic orbits}\quad U
  &=\left\{\sum_{i=1}^4\alpha_i l_i:\alpha_1\neq0\right\},\\
\mbox{Cross section}\quad V_T\cap U
  &=\{\delta l_1+\beta l_3:\delta\neq0,\ \beta\in\mathbb R\}.
\end{align*}
The Pfaffian is a polynomial on $\mathfrak{g}^\ast$ satisfying
$\operatorname{Pf}(l)^2=l(Z)^2$; hence one may choose
$\operatorname{Pf}(l)=l(Z)$.  Identifying $V_T$ with $\mathbb R^2$ and using
normalized Lebesgue measure $d\delta\,d\beta$, the Fourier inversion formula is
\begin{equation}
\phi(\sigma)
=
\int_{\mathbb R^2}
\operatorname{Tr}\!\left(
\pi_{\delta,\beta}(\sigma^{-1})\pi_{\delta,\beta}(\phi)
\right)
|\delta|\,d\delta\,d\beta,
\qquad \phi\in\mathcal S(E_4).
\label{engelplancherel}
\end{equation}

\paragraph{Associated hypo--elliptic operator}

Let $E_4$ denote the Engel group.
For a generic representation $\pi_{\delta,\beta}$,
the corresponding hypo--elliptic operator
acting on $L^2(\mathbb R)$
is given by
\[
\mathcal H
=
-\frac{d^2}{dt^2} + t^4,
\]
a (modified) quartic oscillator.
While the spectrum is discrete,
closed formulas for the eigenvalues are not known.

\subsubsection[The group of 4 x 4 upper triangular matrices N4]{The group of $4\times4$ upper triangular matrices $N_4$}
We summarize Examples 1.1.4, 1.2.6, 1.3.11, 2.2.8, and 4.3.13 in \cite{CorwinGreenleaf}.

\paragraph{Definition}
Let $\mathfrak{g} = \mathfrak{n}_4$ be the Lie algebra of strictly upper triangular $4 \times 4$  matrices; it is a $3$ step nilpotent algebra, of dimension $6$, and its center is one-dimensional. A typical element $X$ can be written as
\begin{align} X &= (z,y_1,y_2,x_1,x_2,x_3)= zZ +y_1Y_1 +y_2Y_2+x_1X_1+x_2X_2+x_3X_3 \notag	\\
&= \left( \begin{array}{cccc}
0 & x_1 & y_1 & z \\
0 & 0   & x_2 & y_2 \\
0 & 0 & 0 & x_3 \\
0 & 0 & 0 & 0
\end{array}\right) \label{liealgebran4}
\end{align}
The calculation of $W_1 \ast W_2$ ($W_1, W_2 \in \mathfrak{g}$) will not be useful for us, and we omit it. We shall, however, compute ${\rm Ad}$:
\begin{align*}
&{}({\rm Ad}( \exp(z, y_1, y_2, x_1, x_2, x_3)))(z', y'_1, y'_2, x'_1, x'_2, x'_3) \\
&= (z'+x_1y'_2-x'_1y_2+y_1x'_3 -  y'_1x_3 + \frac12(x_1x_2x'_3-2x_1x'_2x_3 + x'_1x_2x_3),\\
&{}  y'_1+x_1x'_2 -x_2x'_1, y'_2 + x_2x'_3-x_3x'_2,x'_1,x'_2,x'_3 )
\end{align*}

In the above, $Z,\ldots, X_3$ is a basis of $\mathfrak{n}_4$ corresponding to the matrix entries. The nontrivial commutators are easily computed:
\[
[X_1,X_2] = Y_1 ,\; [X_3,X_2] = -Y_2,\;	[X_1,Y_2] = Z,\; [X_3,Y_1] = - Z.
\]
Since $\mathfrak{a}=\mathbb{R}\mbox{-span}\{Z, Y_1, Y_2, X_2\}$ is the largest abelian ideal in $\mathfrak{n}_4$, $\{Z,Y_1,Y_2,X_2,X_1,X_3\}$ is a strong Malcev basis for $\mathfrak{n}_4$.  

Similarly as (\ref{liealgebran4}), elements $l$ of $\mathfrak{n}_4^\ast$
are written as 
\begin{align} l &= (\alpha,\beta_1,\beta_2,\gamma_1,\gamma_2,\gamma_3) \notag \\ &= \alpha l_1+\beta_1l_2+\beta_2l_3+\gamma_1l_4+\gamma_2l_5+\gamma_3l_6 \notag \\
&= \left( \begin{array}{cccc}
0 & 0 & 0 & 0 \\
\gamma_1 & 0 & 0 & 0 \\
\beta_1 & \gamma_2 & 0 & 0 \\
\alpha & \beta_2 & \gamma_3 & 0
\end{array}\right) \label{duallien4} 
\end{align}
where $\{l_1,\ldots, l_6\}$ is the dual basis of $\{Z,\ldots, X_3\}$.

\paragraph{Coadjoint orbits}
In matrix form, the duality is given by 
\[l(X) = \alpha z + \cdots + \gamma_3x_3.\]
 By the computation of the coadjoint action of $N_4$ on $\mathfrak{n}_4$ in Example 1.3.11 in \cite{CorwinGreenleaf}, we see that the orbits have representatives and parametrizations as follows.
\begin{align*} {\rm (i)}\quad \mathcal{O}_{\alpha',\gamma'_2} &= {\rm Ad}^\ast G(\alpha', 0, 0, 0, \gamma'_2, 0) \\
&= \left\{\left.\left(\alpha', t_1, t_2, s_1, \gamma'_2 + \frac{t_1t_2}{\alpha'},s_2\right)\right| t_1,t_2,s_1,s_2 \in \mathbb{R}\right\} \\
&= \{l| \alpha = \alpha',\alpha\gamma'_2 - \beta_1\beta_2 = \alpha'\gamma_2\}
\end{align*}
where $\alpha' \neq 0$, $\gamma'_2 \in \mathbb{R}$. (These are the orbits of $l$ such that $\alpha \neq 0$).

\begin{align*}{\rm (ii)}\quad \mathcal{O}_{\beta'_1,\beta'_2,\gamma'_3} &=
\mbox{Ad}^\ast G(0, \beta'_1, \beta'_2,0, 0,\gamma'_3) \\
&= \left\{\left.\left(0,\beta'_1,\beta'_2,t_1,t_2, \frac{\beta'_1\gamma'_3-t_1\beta'_2}{\beta'_1}\right) \right| t_1,t_2 \in \mathbb{R}\right\} \\
&= \{l| \alpha=0, \beta_1= \beta'_1, \beta_2 = \beta'_2,\gamma_1\beta_2 +\gamma_3\beta_1 = \beta'_1\gamma'_3 \}
\end{align*}
where $\beta'_1 \neq 0$, $\beta'_2, \gamma'_3 \in \mathbb{R}$. (These are the orbits of $l$ such that $\alpha=0, \beta_1 \neq 0$.) 

\begin{align*}{\rm (iii)}\quad \mathcal{O}_{\beta'_2,\gamma'_1} &=
\mbox{Ad}^\ast G(0, 0, \beta'_2,\gamma'_1, 0,0) \\
&= \{(0,0,\beta'_2,\gamma'_1,t_1,t_2) | t_1,t_2 \in \mathbb{R}\} \\
&= \{l|\alpha= \beta_1 = 0, \beta_2 = \beta'_2, \gamma_1 = \gamma'_1\}
\end{align*}
where $\beta'_2 \neq 0$, $\gamma'_1 \in \mathbb{R}$. (These are the orbits of $l$ such that $\alpha= \beta_1 = 0, \beta_2 \neq 0$.)

\[{\rm (iv)}\quad \mathcal{O}_{\gamma'_1,\gamma'_2,\gamma'_3} =
\mbox{Ad}^\ast G(0, 0, 0,\gamma'_1, \gamma'_2, \gamma'_3) = \{(0, 0, 0,\gamma'_1, \gamma'_2, \gamma'_3)\} 
\]
where $\gamma'_1,\gamma'_2,\gamma'_3 \in \mathbb{R}$. (These are the orbits of $l$ such that $\alpha= \beta_1 = \beta_2 = 0$.)
\paragraph{Irreducible representations associated with generic orbits}
Generic orbits in (i) are $4$-dimensional hypersurfaces; those in (ii) and (iii) are
$2$-dimensional planes, and those in (iv) are points. 
Taking the indicated orbit representatives $l=l_{\alpha',\gamma'_2}$, in Case (i), $\mathfrak{r}_l = \{(z, 0, 0, 0, x_2 , 0)| z, x_2 \in \mathbb{R}\}$ and the abelian ideal $\mathfrak{m} =\{(z, y_1, y_2, 0, x_2, 0)| z, y_1,y_2, x_2 \in \mathbb{R}\}$ is a polarizing subalgebra for each such $l$. 
Since $\mathfrak{m}$ is an ideal, we may use the exponential coordinates
$(z,\ldots, x_3) = \exp X$ in $G = N_4$ and obtain a cross-section for $M\backslash G $,
$\Sigma = \{(0,0,0,t_1,0,t_3)| t_1,t_3 \in \mathbb{R}\}$. The action of $\pi_{\alpha',\gamma'_2}$ can then  be modeled in $L^2(\mathbb{R}^2)$. 
The necessary splitting of the product, computed  using the Campbell-Hausdorff formula, is
\[
(0,0,0,t_1,0,t_3)\cdot (z,\ldots, x_3)=(z',y'_1, y'_2,0,x_2,0)\cdot (0,0,0,t_1+x_1,0,t_3 +x_3),
\]
where
\begin{align*}
z' &= z + \frac{x_1y_2}{2}-\frac{x_3y_1}{2} -t_1x_2t_3-\frac12x_1x_2t_3 -\frac12t_1x_2x_3+\frac13x_1x_2x_3 \\
y'_1 &= y_1 + t_1x_2+\frac{x_1x_2}{2} \\
y'_2 &= y_2 - t_3x_2+\frac{x_2x_3}{2}.
\end{align*}
Thus, the action on $\tilde{f} \in L^2(\mathbb{R}^2)$ is
\begin{align}
&{}[\pi_{\alpha',\gamma'_2}\tilde{f}](t_1,t_3) \notag \\
&= e^{2\pi\sqrt{-1}\gamma'_2x_2}e^{2\pi\sqrt{-1}\alpha'[z + (1/2)x_1y_2 -x_3y_1 -t_1x_2t_3-(1/2)x_1x_2t_3 -(1/2)t_1x_2x_3+(1/3)x_1x_2x_3]} \notag \\
&{} \cdot\tilde{f}(t_1+x_1,t_3 +x_3). \label{4upperirred}
\end{align}
The above computation is enough for our later arguments, and thus, we proceed to the next step. 
(The calculations of the nongeneric representations are left to the reader in Example 2.2.8. in \cite{CorwinGreenleaf}).

If we choose the dual basis in (\ref{duallien4}), we have
\[ 
S=\{2,3,5,6\},\quad T=\{1,4\}
\]
and
\begin{align*}
\mbox{Generic orbits}\quad  U &= \{\alpha l_1+\beta_1l_2+\beta_2l_3+\gamma_1l_4+\gamma_2l_5+\gamma_3l_6 | \alpha \neq 0 \},\\
\mbox{Cross section} V_T \cap U &= \{\alpha l_1 + \gamma_1l_4 | \alpha \neq 0 \}.
\end{align*}

The Pfaffian $\mbox{Pf}(l)$ is expressed as 
\[
\mbox{Pf}(l)^2 = \mbox{det} B_l(X_{i_j}, X_{i_k}) = \mbox{det}\left(\begin{array}{cccc}0 & 0 & 0 & l(Z) \\
0 & 0 & l(Z) & 0 \\
0 & l(Z) & 0 & 0 \\
l(Z) & 0 & 0 & 0
\end{array}\right)
\]

\paragraph{Fourier inversion formula}
Identifying $V_T\cong\mathbb R^2$ and using normalized Lebesgue measure
$d\alpha\,d\gamma_1$, the Fourier inversion and Plancherel formulas are
\begin{align}
\phi(\sigma)
&=
\int_{\mathbb R^2}
\operatorname{Tr}\!\left(
\pi_{\alpha,\gamma_1}(\sigma^{-1})
\pi_{\alpha,\gamma_1}(\phi)
\right)
|\alpha|^2\,d\alpha\,d\gamma_1,
\qquad \phi\in\mathcal S(N_4),
\notag\\
\|f\|_2^2
&=
\int_{\mathbb R^2}
\|\pi_{\alpha,\gamma_1}(f)\|_{\rm HS}^2
|\alpha|^2\,d\alpha\,d\gamma_1,
\qquad f\in\mathcal S(N_4).
\label{4upperplancherel}
\end{align}
Here $\|\cdot\|_{\rm HS}$ denotes the Hilbert--Schmidt norm.

\paragraph{Associated hypo--elliptic operator}

For the nilpotent group of $4\times4$ upper triangular matrices,
the representation space is $L^2(\mathbb R^2)$,
and the canonical operator becomes
\[
\mathcal H
=
-\left(
\frac{\partial^2}{\partial t_1^2}
+
\frac{\partial^2}{\partial t_3^2}
\right)
+
4\pi^2 (t_1 t_3)^2.
\]
This operator provides an example of a
polynomial Schr\"odinger operator
arising from a step--three nilpotent group.

\section{Remarks on trace class issues and formal manipulations}
\label{AppendixD}

In the passage from discrete groups to their Malcev completions one
frequently encounters operators that are formally defined on
infinite-dimensional representation spaces but are not trace class.  This
phenomenon already appears for the Heisenberg group.  If a function on
$\Gamma$ is extended to a function on $G$, the operator $\pi_l(f)$ in a
Kirillov representation may be compact without being trace class.

For this reason the trace formulas in the lattice theory should not be read
as ordinary traces of infinite-dimensional Lie-group operators unless the
required trace-class hypotheses are verified separately.  The role of the
finite-dimensional refinement is to replace such formal traces, at the
lattice level, by genuine matrix traces.  The Pytlik-type formula in
Theorem~\ref{newPytlik} is therefore a finite-dimensional normalized-trace
identity supported by a finitely additive measure.  It is not a Hilbert-space
Plancherel decomposition of the full non--type~I dual.

This distinction is also useful analytically.  Lie-group representations
provide smooth differential models, such as the harmonic oscillator in the
Heisenberg case and higher polynomial Schr\"odinger operators for more
general nilpotent groups.  The lattice theorem explains how these models are
related to exact finite-dimensional fibers.  Approximation or asymptotic
analysis enters only after this exact representation-theoretic comparison
has been made.

\section{Finite quotients and the finitely additive inversion functional}
\label{finiteadditiveappendix}

We give the finite-quotient construction behind Theorem~\ref{newPytlik} in
some detail.  The construction is elementary, but it is the cleanest way to
separate three notions: finite-group Fourier inversion, the ultralimit which
produces a finitely additive measure, and the finite-dimensional representation-theoretic description of
particular finite-dimensional fibers.

\begin{remark}[Why the finite-quotient model is recorded]
The main text introduces the Pytlik-type functional geometrically on the
rational finite-dimensional skeleton, in the spirit of Pytlik's original construction.  The
geometric content is defined by starting with ordinary regular sets in the
ambient parameter torus: if \(D\subset T\) is the rational finite-dimensional skeleton and
\(B\subset T\) is a finite union of half-open cubes, or more generally a
Jordan set with Lebesgue-null boundary, then the content of \(D\cap B\) is
\(\operatorname{Leb}_T(B)\).  This cube-first convention is finitely additive
on the corresponding algebra of regular rational sets and is the natural
generalization of the Heisenberg Pytlik measure.

The finite-quotient construction below is recorded for a different reason.
Starting from ordinary Fourier inversion on finite quotients automatically
supplies positivity, total mass one, the noncommutative Plancherel weights
\((\dim\rho)^2/|Q|\), and the normalized trace identity.  It also extends the
coefficient functional to all bounded observables needed for the inversion
argument.

Thus the finite-quotient ultralimit is not meant to replace the geometric
intuition of the rational finite-dimensional skeleton.  It gives another
exact realization of the coefficient functional on the algebra used in the
inversion formula.  No literal equality with the geometric cube content as a
finitely additive set function is asserted, and no canonicity with respect to
the residual chain or free ultrafilter is claimed.  What is common to the two
models is the inversion value on the relevant coefficient functions and the
normalized-trace convention.
\end{remark}

\subsection{Residual finiteness}

A finitely generated nilpotent group is residually finite.  Hence one can
choose a decreasing sequence of normal finite-index subgroups
\[
\Gamma=N_0\supset N_1\supset N_2\supset\cdots,
\qquad
\bigcap_{j=0}^{\infty}N_j=\{e\}.
\]
Put $Q_j=\Gamma/N_j$ and let $q_j:\Gamma\to Q_j$ be the quotient map.  If
$f\in\mathbb C[\Gamma]$ has finite support and $\sigma\in\Gamma$ is fixed,
then for all sufficiently large $j$ the quotient map separates all elements
needed to compute the coefficient $f(\sigma)$.  Consequently that coefficient
can be recovered from the image of $f$ in the finite group algebra
$\mathbb C[Q_j]$.

\subsection{Finite-group Fourier inversion}

For a finite group $Q$ and $F\in\mathbb C[Q]$, standard matrix-coefficient
orthogonality gives
\[
F(y)
=
\sum_{\rho\in\widehat Q}
\frac{\dim\rho}{|Q|}
\mathrm{Tr}\bigl(\rho(y^{-1})\rho(F)\bigr).
\]
Equivalently, with the probability weights
\[
m_Q(\rho)=\frac{(\dim\rho)^2}{|Q|},
\]
this becomes
\[
F(y)
=
\sum_{\rho\in\widehat Q}
\frac1{\dim\rho}
\mathrm{Tr}\bigl(\rho(y^{-1})\rho(F)\bigr)m_Q(\rho).
\]
All sums here are finite and all traces are ordinary matrix traces.

Pull the representations of $Q_j$ back to $\Gamma$.  For a bounded function
$F$ on $\widehat\Gamma_{\mathrm{fin}}$, set
\[
\Lambda_j(F)
=
\sum_{\rho\in\widehat{Q_j}}F(\rho)\frac{(\dim\rho)^2}{|Q_j|}.
\]
The functionals $\Lambda_j$ are positive and have norm one.  If
$\mathcal U$ is a free ultrafilter on $\mathbb N$, define
\[
\Lambda(F)=\lim_{j\to\mathcal U}\Lambda_j(F).
\]
Then $\Lambda$ is again a positive norm-one functional.  Equivalently, it is
integration against a positive finitely additive probability measure on the
Boolean algebra of subsets of $\widehat\Gamma_{\mathrm{fin}}$.

\subsection{Exactness of the ultralimit identity}

Let $f\in\mathbb C[\Gamma]$ and $\sigma\in\Gamma$.  For all sufficiently
large $j$, finite-group inversion on $Q_j$ gives exactly
\[
f(\sigma)
=
\sum_{\rho\in\widehat{Q_j}}
\frac1{\dim\rho}
\mathrm{Tr}\bigl(\rho(\sigma^{-1})\rho(f)\bigr)
\frac{(\dim\rho)^2}{|Q_j|}.
\]
Taking the ultralimit gives
\[
f(\sigma)
=
\Lambda\left(
\rho\mapsto
\frac1{\dim\rho}
\mathrm{Tr}\bigl(\rho(\sigma^{-1})\rho(f)\bigr)
\right).
\]
If $f\in\ell^1(\Gamma)$, approximate it by finitely supported functions.  The
estimate
\[
\left|\frac1{\dim\rho}
\mathrm{Tr}\bigl(\rho(\sigma^{-1})\rho(g)\bigr)\right|
\leq \|g\|_1
\]
passes the identity to $\ell^1(\Gamma)$.

\subsection{Large height and finite-dimensional parameters}

Define the height of a finite-dimensional representation pulled back from
this tower to be the least $j$ through which it factors.  For fixed $J$, the
representations of height at most $J$ come from the fixed finite group
$Q_J$.  Their total Plancherel weight inside $Q_j$ is bounded by
$|Q_J|/|Q_j|$, which tends to zero.  Thus every bounded-height subset has
finitely additive measure zero.

For the Heisenberg lattice this recovers the familiar large-denominator
feature of Pytlik's measure: rational central characters with bounded
denominator have measure zero.  For general nilpotent lattices, height is the
corresponding invariant.  The classification theorem of Howe identifies the finite-dimensional fibers
arising from rational coadjoint data; the finite-quotient construction gives
the finitely additive averaging functional for their normalized traces.  This
appendix records only the trace-functional convention.  Quantitative analytic
comparison at large denominator is not proved here; its required topology and
class of observables must be specified separately, as explained in
Remarks~\ref{remark34} and in the large-denominator remark following
Theorem~\ref{newPytlik}.


\address{ 
Research and Education Center for Natural Science \\
Hiyoshi Campus, Keio University \\
4-1-1, Hiyoshi, Kohoku-ku \\
 Yokohama 223-8521, Japan \\
\smallskip \\
Osaka Central Advanced Mathematical Institute (OCAMI) \\
MEXT Joint Usage/Research Center on Mathematics and Theoretical Physics, \\
Osaka Metropolitan University \\
3-3-138 Sugimoto, Sumiyoshi-ku \\
Osaka 558-8585, Japan \\
\smallskip \\
Faculty of Mathematics \\
Kyushu University \\
744 Motooka, Nishi-ku,\\
Fukuoka 819-0395, Japan
}
{katsuda@math.kyushu-u.ac.jp}


\addcontentsline{toc}{section}{References}
\begin{thebibliography}{99}

%


\bibitem{Alexopoulos3}
\textsc{G. Alexopoulos},
Sub-Laplacians with drift on Lie groups of polynomial volume growth,
Mem. Amer. Math. Soc. 155 (2002), no. 739.

\bibitem{AvilaJitomirskaya}
A. Avila and S. Jitomirskaya, \emph{The Ten Martini Problem}, Ann. of Math. (2) \textbf{170} (2009), no. 1, 303--342.

\bibitem{BekkaDriutti}
B. Bekka and P. Driutti, \emph{Restrictions of irreducible unitary representations of nilpotent Lie groups to lattices}, J. Funct. Anal. \textbf{168} (1999), no. 2, 514--528.

\bibitem{BekkaPlancherel}
B. Bekka, \emph{The Plancherel formula for countable groups}, Indag. Math. (N.S.) \textbf{32} (2021), no. 3, 619--638.

\bibitem{ChoiElliottYui}
M. D. Choi, G. A. Elliott, and N. Yui, \emph{Gauss polynomials and the rotation algebra}, Invent. Math. \textbf{99} (1990), no. 2, 225--246.

\bibitem{Corwin2}
\textsc{L.~Corwin, F.~P.~Greenleaf and G.~Gr\'elaud},
Direct integral decompositions and multiplicities for induced representations
of nilpotent Lie groups,
\emph{Trans. Amer. Math. Soc.} \textbf{304} (1987), 549--583.

\bibitem{CorwinGreenleaf}
L. J. Corwin and F. P. Greenleaf, \emph{Representations of Nilpotent Lie Groups and Their Applications, Part I}, Cambridge Studies in Advanced Mathematics, vol. 18, Cambridge University Press, Cambridge, 1990.

\bibitem{Davidson}
\textsc{K. R. Davidson},
$C^\ast$-algebras by example, Fields Institute Monographs, 6. Amer. Math. Soc., 1996.

\bibitem{Fujiwara}
\textsc{H. Fujiwara and J. Ludwig},
Harmonic analysis on exponential solvable Lie groups, Springer Monographs in Mathematics. Springer, Tokyo, 2015. xii+465 pp.

\bibitem{FujiwaraNote}
H. Fujiwara, unpublished note communicated to the author, 2018.

\bibitem{FujiwaraPlancherel}
H. Fujiwara, \emph{La formule de Plancherel pour les repr\'esentations monomiales des groupes de Lie nilpotents}, in \emph{Representation Theory of Lie Groups and Lie Algebras} (Fuji--Kawaguchiko, 1990), World Scientific, Singapore, 1992, pp. 140--150.

\bibitem{Glimm}
\textsc{J.~Glimm},
Type~I $C^*$-algebras,
\emph{Ann. of Math.} (2) \textbf{73} (1961), 572--612.

\bibitem{Gruber}
\textsc{M.~J.~Gruber},
Noncommutative Bloch theory,
\emph{J. Math. Phys.} \textbf{42} (2001), no.~6, 2438--2465.

\bibitem{HelfferSjostrandI}
B. Helffer and J. Sj\"ostrand, \emph{Analyse semi-classique pour l'\'equation de Harper (avec application \`a l'\'equation de Schr\"odinger avec champ magn\'etique)}, M\'em. Soc. Math. France (N.S.) No. 34 (1988), 1--113.

\bibitem{HelfferSjostrandII}
B. Helffer and J. Sj\"ostrand, \emph{Analyse semi-classique pour l'\'equation de Harper. II: Comportement semi-classique pr\`es d'un rationnel}, M\'em. Soc. Math. France (N.S.) No. 40 (1990).

\bibitem{Howe}
\textsc{R.~E.~Howe},
On representations of discrete, finitely generated,
torsion-free nilpotent groups,
\emph{Pacific J. Math.} \textbf{73} (1977), no.~2, 281--305.


\bibitem{Johnston}
\textsc{C.~P.~Johnston},
On a Plancherel formula for certain discrete, finitely generated,
torsion-free nilpotent groups,
\emph{Pacific J. Math.} \textbf{167} (1995), no.~2, 313--326.

\bibitem{Kirillov1962}
A. A. Kirillov, \emph{Unitary representations of nilpotent Lie groups}, Uspekhi Mat. Nauk \textbf{17} (1962), no. 4, 57--110.

\bibitem{Kirillov2004}
A. A. Kirillov, \emph{Lectures on the Orbit Method}, Graduate Studies in Mathematics, vol. 64, American Mathematical Society, Providence, RI, 2004.

\bibitem{Kirillov3}
\textsc{A. A. Kirillov},
Thoughts about George Mackey and his imprimitivity theorem,
Contemp. Math., 449,
American Mathematical Society, Providence, RI, 2008, 247--262.

\bibitem{Lubotzky}
\textsc{A. Lubotzky and Y. Shalom},
Finite representations in the unitary dual and Ramanujan groups,
in \textit{Discrete geometric analysis}, Contemp. Math. 347,
Amer. Math. Soc., Providence, RI, 2004, 173--189.

\bibitem{Mackey}
G. W. Mackey, \emph{Induced Representations of Groups and Quantum Mechanics}, W. A. Benjamin, New York, 1968.

\bibitem{MathaiMarcolli}
\textsc{V.~Mathai and M.~Marcolli},
Twisted index theory on good orbifolds, I: Noncommutative Bloch theory,
\emph{Commun. Contemp. Math.} \textbf{1} (1999), no.~4, 553--587.

\bibitem{Mautner}
F. I. Mautner, \emph{The structure of the regular representation of certain discrete groups}, Duke Math. J. \textbf{17} (1950), 437--441.

\bibitem{Pytlik}
\textsc{T.~Pytlik},
A Plancherel measure for the discrete Heisenberg group,
\emph{Colloq. Math.} \textbf{42} (1979), 355--359.

\bibitem{RammalBellissard}
R. Rammal and J. Bellissard, \emph{An algebraic semiclassical approach to Bloch electrons in a magnetic field}, J. Phys. France \textbf{51} (1990), 1803--1830.

\bibitem{Sakurai}
J. J. Sakurai and J. Napolitano, \emph{Modern Quantum Mechanics}, 2nd ed., Addison--Wesley, San Francisco, 2011.

\bibitem{Segal}
I. E. Segal, \emph{An extension of Plancherel's formula to separable unimodular groups}, Ann. of Math. (2) \textbf{52} (1950), 272--292.

\bibitem{ThomaRegulareMass}
E. Thoma, \emph{\"Uber das regul\"are Ma\ss{} im dualen Raum diskreter Gruppen}, Math. Z. \textbf{100} (1967), 257--271.

\bibitem{ThomaUnitary}
E. Thoma, \emph{\"Uber unit\"are Darstellungen abz\"ahlbarer, diskreter Gruppen}, Math. Ann. \textbf{153} (1964), 111--138.

\bibitem{Vergne}
\textsc{M.~Vergne},
Construction de sous-alg\`ebres subordonn\'ees
\`a un \'element du dual d'une alg\`ebre de Lie r\'esoluble,
\emph{C. R. Acad. Sci. Paris} S\'er.~A--B \textbf{270} (1970),
A173--175.

\bibitem{Wilkinson}
\textsc{M.~Wilkinson},
An example of phase holonomy in WKB theory,
\emph{J. Phys. A} \textbf{17} (1984), 3459--3476.


\end{thebibliography}
\end{document}